\documentclass[11pt]{preprint}
\usepackage{difftrees} 
\usepackage[full]{textcomp}
\usepackage[osf]{newtxtext} 
\usepackage[cal=boondoxo]{mathalfa}
\usepackage{colortbl}

\usepackage{tikz-cd} 
\usetikzlibrary{cd} 
\usepackage[all,cmtip]{xy}
\usepackage{comment}

\usepackage{amssymb}
\usepackage{mathtools}
\usepackage{hyperref}
\usepackage{breakurl}
\usepackage{mhenvs}
\usepackage{mhequ} 
\usepackage{mhsymb}
\usepackage{booktabs}
\usepackage{tikz}
\usepackage{tcolorbox}
\usepackage{mathrsfs}
\usepackage[utf8]{inputenc}
\usepackage{longtable}
\usepackage{wrapfig}
\usepackage{rotating} 
\usepackage{subcaption}
\usepackage{mathrsfs}
\usepackage{epsfig}
\usepackage{microtype}
\usepackage{comment}
\usepackage{wasysym}
\usepackage{centernot}
\usepackage{enumitem}
\usepackage{bm}
\usepackage{stackrel}

\usepackage{graphicx}
\usepackage{axodraw}
\usepackage{xspace}
\usepackage{subcaption} 
\usepackage{epsfig} 
\usepackage{axodraw} 
\usepackage{xspace} 
\usepackage[toc,page]{appendix} 
\usepackage[all,cmtip]{xy} 
\usepackage{relsize} 
\usepackage{shuffle} 
\makeatletter
\newcommand{\globalcolor}[1]{%
  \color{#1}\global\let\default@color\current@color
}
\makeatother

\usetikzlibrary{calc}
\usetikzlibrary{decorations}
\usetikzlibrary{positioning}
\usetikzlibrary{shapes}
\usetikzlibrary{external}

\definecolor{blush}{rgb}{0.87, 0.36, 0.51}
	\definecolor{brightcerulean}{rgb}{0.11, 0.67, 0.84}
	\definecolor{greenryb}{rgb}{0.4, 0.69, 0.2}

\newif\ifdark
\darkfalse

\ifdark
\definecolor{darkred}{rgb}{0.9,0.2,0.2}
\definecolor{darkblue}{rgb}{0.7,0.3,1}
\definecolor{darkgreen}{rgb}{0.1,0.9,0.1}
\definecolor{franck}{rgb}{0,0.8,1}
\definecolor{pagebackground}{rgb}{.15,.21,.18}
\definecolor{pageforeground}{rgb}{.84,.84,.85}
\pagecolor{pagebackground}
\AtBeginDocument{\globalcolor{pageforeground}}
\definecolor{symbols}{rgb}{0,0.7,1}
\colorlet{connection}{red!80!black}
\colorlet{boxcolor}{blue!50}

\else

\definecolor{darkred}{rgb}{0.7,0.1,0.1}
\definecolor{darkblue}{rgb}{0.4,0.1,0.8}
\definecolor{darkgreen}{rgb}{0.1,0.7,0.1}
\definecolor{franck}{rgb}{0,0,1}
\definecolor{pagebackground}{rgb}{1,1,1}
\definecolor{pageforeground}{rgb}{0,0,0}
\colorlet{symbols}{blue!90!black}
\colorlet{connection}{red!30!black}
\colorlet{boxcolor}{blue!50!black}

\fi

\def\slash{\leavevmode\unskip\kern0.18em/\penalty\exhyphenpenalty\kern0.18em}
\def\dash{\leavevmode\unskip\kern0.18em--\penalty\exhyphenpenalty\kern0.18em}

\DeclareMathAlphabet{\mathbbm}{U}{bbm}{m}{n}

\DeclareFontFamily{U}{BOONDOX-calo}{\skewchar\font=45 }
\DeclareFontShape{U}{BOONDOX-calo}{m}{n}{
  <-> s*[1.05] BOONDOX-r-calo}{}
\DeclareFontShape{U}{BOONDOX-calo}{b}{n}{
  <-> s*[1.05] BOONDOX-b-calo}{}
\DeclareMathAlphabet{\mcb}{U}{BOONDOX-calo}{m}{n}
\SetMathAlphabet{\mcb}{bold}{U}{BOONDOX-calo}{b}{n}

\setlist{noitemsep,topsep=4pt,leftmargin=1.5em}

\DeclareMathAlphabet{\mathbbm}{U}{bbm}{m}{n}

\DeclareMathAlphabet{\mcb}{U}{BOONDOX-calo}{m}{n}
\SetMathAlphabet{\mcb}{bold}{U}{BOONDOX-calo}{b}{n}
\DeclareFontFamily{U}{mathx}{\hyphenchar\font45}
\DeclareFontShape{U}{mathx}{m}{n}{
      <5> <6> <7> <8> <9> <10>
      <10.95> <12> <14.4> <17.28> <20.74> <24.88>
      mathx10
      }{}
\DeclareSymbolFont{mathx}{U}{mathx}{m}{n}
\DeclareMathSymbol{\bigtimes}{1}{mathx}{"91}

\def\s{\mathfrak{s}}

\providecommand{\figures}{false}
{ \ifthenelse{\equal{\figures}{false}} {#1}{\[ {\rm Figure \ missing !} \]} }{}

\def\CH{\mathcal{H}}

\def\CT{\mathcal{T}}

\tikzstyle{tinydots}=[dash pattern=on \pgflinewidth off \pgflinewidth]
\tikzstyle{superdense}=[dash pattern=on 4pt off 1pt]

\newcommand{\beq}{\begin{equation}}
\newcommand{\eeq}{\end{equation}}

\usepackage{empheq}

\newcommand{\mfT}{\mathfrak{T}}

\def\Labe{\mathfrak{e}}

\def\Labn{\mathfrak{n}}

\def\${|\!|\!|}

\newcounter{theorem}

\newtheorem{ex}[theorem]{Example}

\newenvironment{DIFnomarkup}{}{} 

\theorembodyfont{\rmfamily}

\newcommand{\rrightarrow}{{\to\hskip -4.9mm\raise 1pt\hbox{$\to$}}}

\newfont{\indic}{bbmss12}

\def\Nabla_#1{\nabla_{\!#1}}

\makeatletter
\pgfdeclareshape{crosscircle}
{
  \inheritsavedanchors[from=circle] 
  \inheritanchorborder[from=circle]
  \inheritanchor[from=circle]{north}
  \inheritanchor[from=circle]{north west}
  \inheritanchor[from=circle]{north east}
  \inheritanchor[from=circle]{center}
  \inheritanchor[from=circle]{west}
  \inheritanchor[from=circle]{east}
  \inheritanchor[from=circle]{mid}
  \inheritanchor[from=circle]{mid west}
  \inheritanchor[from=circle]{mid east}
  \inheritanchor[from=circle]{base}
  \inheritanchor[from=circle]{base west}
  \inheritanchor[from=circle]{base east}
  \inheritanchor[from=circle]{south}
  \inheritanchor[from=circle]{south west}
  \inheritanchor[from=circle]{south east}
  \inheritbackgroundpath[from=circle]
  \foregroundpath{
    \centerpoint%
    \pgf@xc=\pgf@x%
    \pgf@yc=\pgf@y%
    \pgfutil@tempdima=\radius%
    \pgfmathsetlength{\pgf@xb}{\pgfkeysvalueof{/pgf/outer xsep}}%
    \pgfmathsetlength{\pgf@yb}{\pgfkeysvalueof{/pgf/outer ysep}}%
    \ifdim\pgf@xb<\pgf@yb%
      \advance\pgfutil@tempdima by-\pgf@yb%
    \else%
      \advance\pgfutil@tempdima by-\pgf@xb%
    \fi%
    \pgfpathmoveto{\pgfpointadd{\pgfqpoint{\pgf@xc}{\pgf@yc}}{\pgfqpoint{-0.707107\pgfutil@tempdima}{-0.707107\pgfutil@tempdima}}}
    \pgfpathlineto{\pgfpointadd{\pgfqpoint{\pgf@xc}{\pgf@yc}}{\pgfqpoint{0.707107\pgfutil@tempdima}{0.707107\pgfutil@tempdima}}}
    \pgfpathmoveto{\pgfpointadd{\pgfqpoint{\pgf@xc}{\pgf@yc}}{\pgfqpoint{-0.707107\pgfutil@tempdima}{0.707107\pgfutil@tempdima}}}
    \pgfpathlineto{\pgfpointadd{\pgfqpoint{\pgf@xc}{\pgf@yc}}{\pgfqpoint{0.707107\pgfutil@tempdima}{-0.707107\pgfutil@tempdima}}}
  }
}
\makeatother

\def\symbol#1{\textcolor{symbols}{#1}}

\def\decorate#1#2{
        \ifnum#2>0
    		\foreach \count in {1,...,#2}{
	       	let
				\p1 = (sourcenode.center),
                \p2 = (sourcenode.east),
				\n1 = {\x2-\x1},
				\n2 = {1mm},
				\n3 = {(1.3+0.6*(\count-1))*\n1},
				\n4 = {0.7*\n1}
			in 
        		node[rectangle,fill=symbols,rotate=30,inner sep=0pt,minimum width=0.2*\n2,minimum height=\n2] at ($(sourcenode.center) + (\n3,\n4)$) {}
				}
		\fi
        \ifnum#1>0
    		\foreach \count in {1,...,#1}{
	       	let
				\p1 = (sourcenode.center),
                \p2 = (sourcenode.east),
				\n1 = {\x2-\x1},
				\n2 = {1mm},
				\n3 = {(1.3+0.6*(\count-1))*\n1},
				\n4 = {0.7*\n1}
			in 
        		node[rectangle,fill=symbols,rotate=-30,inner sep=0pt,minimum width=0.2*\n2,minimum height=\n2] at ($(sourcenode.center) + (-\n3,\n4)$) {}
				}
		\fi
}

\tikzset{
    dectriangle/.style 2 args={
        triangle,
        alias=sourcenode,
        append after command={\decorate{#1}{#2}}
    },
    dectriangle/.default={0}{0},
}

\tikzset{
	cross/.style={path picture={ 
  		\draw[symbols]
			(path picture bounding box.south east) -- (path picture bounding box.north west) (path picture bounding box.south west) -- (path picture bounding box.north east);
		}},
root/.style={circle,fill=green!50!black,inner sep=0pt, minimum size=1.2mm},
        dot/.style={circle,fill=pageforeground,inner sep=0pt, minimum size=1mm},
        dotred/.style={circle,fill=pageforeground!50!pagebackground,inner sep=0pt, minimum size=2mm},
        var/.style={circle,fill=pageforeground!10!pagebackground,draw=pageforeground,inner sep=0pt, minimum size=3mm},
        kernel/.style={semithick,shorten >=2pt,shorten <=2pt},
        kernels/.style={snake=zigzag,shorten >=2pt,shorten <=2pt,segment amplitude=1pt,segment length=4pt,line before snake=2pt,line after snake=5pt,},
        rho/.style={densely dashed,semithick,shorten >=2pt,shorten <=2pt},
           testfcn/.style={dotted,semithick,shorten >=2pt,shorten <=2pt},
        renorm/.style={shape=circle,fill=pagebackground,inner sep=1pt},
        labl/.style={shape=rectangle,fill=pagebackground,inner sep=1pt},
        xic/.style={very thin,circle,draw=symbols,fill=symbols,inner sep=0pt,minimum size=1.2mm},
        g/.style={very thin,rectangle,draw=symbols,fill=symbols!10!pagebackground,inner sep=0pt,minimum width=2.5mm,minimum height=1.2mm},
        xi/.style={very thin,circle,draw=symbols,fill=symbols!10!pagebackground,inner sep=0pt,minimum size=1.2mm},
	xies/.style={very thin,rectangle,fill=green!50!black!25,draw=symbols,inner sep=0pt,minimum size=1.1mm},
	xiesf/.style={very thin,rectangle,fill=green!50!black,draw=symbols,inner sep=0pt,minimum size=1.1mm},
        xix/.style={very thin,crosscircle,fill=symbols!10!pagebackground,draw=symbols,inner sep=0pt,minimum size=1.2mm},
        X/.style={very thin,cross,rectangle,fill=pagebackground,draw=symbols,inner sep=0pt,minimum size=1.2mm},
	xib/.style={thin,circle,fill=symbols!10!pagebackground,draw=symbols,inner sep=0pt,minimum size=1.6mm},
	xie/.style={thin,circle,fill=green!50!black,draw=symbols,inner sep=0pt,minimum size=1.6mm},
	xid/.style={thin,circle,fill=symbols,draw=symbols,inner sep=0pt,minimum size=1.6mm},
	xibx/.style={thin,crosscircle,fill=symbols!10!pagebackground,draw=symbols,inner sep=0pt,minimum size=1.6mm},
	kernels2/.style={very thick,draw=connection,segment length=12pt},
	keps/.style={thin,draw=symbols,->},
	kepspr/.style={thick,draw=connection,->},
	krho/.style={thin,draw=symbols,superdense,->},
	krhopr/.style={thick,draw=connection,superdense},
	triangle/.style = { regular polygon, regular polygon sides=3},
	not/.style={thin,circle,draw=connection,fill=connection,inner sep=0pt,minimum size=0.5mm},
	diff/.style = {very thin,draw=symbols,triangle,fill=red!50!black,inner sep=0pt,minimum size=1.6mm},
	diff1/.style = {very thin,dectriangle={1}{0},fill=red!50!black,draw=symbols,inner sep=0pt,minimum size=1.6mm},
	diff2/.style = {very thin,dectriangle={1}{1},fill=red!50!black,draw=symbols,inner sep=0pt,minimum size=1.6mm},
		diffmini/.style = {very thin,rectangle,fill=black,draw=black,inner sep=0pt,minimum size=0.75mm},
	 kernelsmod/.style={very thick,draw=connection,segment length=12pt},
	 rec/.style = {very thin,rectangle,fill=black,draw=black,inner sep=0pt,minimum size=2mm},
	cerc/.style={very thin,circle,draw=black,fill=symbols,inner sep=0pt,minimum size=2mm},
	stars/.style={very thin,star,star points=6,star point ratio=0.5, draw=black,fill=red,inner sep=0pt,minimum size=0.7mm},
	>=stealth,
        }
        \tikzset{
root/.style={circle,fill=black!50,inner sep=0pt, minimum size=3mm},
        circ/.style={circle,fill=white,draw=black,very thin,inner sep=.5pt, minimum size=1.2mm},
        round1/.style={fill=white,outer sep = 0,inner sep=2pt,rounded corners=1mm,draw,text=black,thin,minimum size=1.2mm},
          circ1/.style={circle,fill=red!10,draw=red,very thin,inner sep=.5pt, minimum size=1.2mm},
        rect/.style={fill=white,outer sep = 0,inner sep=2pt,rectangle,draw,text=black,thin,minimum size=1.2mm},
        rect1/.style={fill=white,outer sep = 0,inner sep=2pt,rectangle,draw,text=black,thin,minimum size=1.2mm},
        round2/.style={fill=red!10,outer sep = 0,inner sep=2pt,rounded corners=1mm,draw,text=black,thin,minimum size=1.2mm},
       round3/.style={fill=blue!10,outer sep = 0,inner sep=2pt,rounded corners=1mm,draw,text=black,thin,minimum size=1.2mm}, 
        rect2/.style={fill=black!10,outer sep = 0,inner sep=2pt,rectangle,draw,text=black,thin,minimum size=1.2mm},
        dot/.style={circle,fill=black,inner sep=0pt, minimum size=1.2mm},
        dotred/.style={circle,fill=black!50,inner sep=0pt, minimum size=2mm},
        var/.style={circle,fill=black!10,draw=black,inner sep=0pt, minimum size=3mm},
        kernel/.style={semithick,shorten >=2pt,shorten <=2pt},
         diag/.style={thin,shorten >=4pt,shorten <=4pt},
        kernel1/.style={thick},
        kernels/.style={snake=zigzag,shorten >=2pt,shorten <=2pt,segment amplitude=1pt,segment length=4pt,line before snake=2pt,line after snake=5pt,},
		kernels1/.style={snake=zigzag,segment amplitude=0.5pt,segment length=2pt},
		rho1/.style={densely dotted,semithick},
        rho/.style={densely dashed,semithick,shorten >=2pt,shorten <=2pt},
           testfcn/.style={dotted,semithick,shorten >=2pt,shorten <=2pt},
           visible/.style={draw, circle, fill, inner sep=0.25ex},
        renorm/.style={shape=circle,fill=white,inner sep=1pt},
        labl/.style={shape=rectangle,fill=white,inner sep=1pt},
        xic/.style={very thin,circle,fill=symbols,draw=black,inner sep=0pt,minimum size=1.2mm},
        xi/.style={very thin,circle,fill=blue!10,draw=black,inner sep=0pt,minimum size=1.2mm},
	xib/.style={very thin,circle,fill=blue!10,draw=black,inner sep=0pt,minimum size=1.6mm},
	xie/.style={very thin,circle,fill=green!50!black,draw=black,inner sep=0pt,minimum size=1mm},
	xid/.style={very thin,circle,fill=symbols,draw=black,inner sep=0pt,minimum size=1.6mm},
	edgetype/.style={very thin,circle,draw=black,inner sep=0pt,minimum size=5mm},
	nodetype/.style={very thick,circle,draw=black,inner sep=0pt,minimum size=5mm},
	kernels2/.style={very thick,draw=connection,segment length=12pt},
clean/.style={thin,circle,fill=black,inner sep=0pt,minimum size=1mm},	not/.style={thin,circle,fill=symbols,draw=connection,fill=connection,inner sep=0pt,minimum size=0.8mm},
	>=stealth,
        }

\makeatletter
\def\DeclareSymbol#1#2#3{%
	\expandafter\gdef\csname MH@symb@#1\endcsname{\tikzsetnextfilename{symbol#1}%
	\tikz[baseline=#2,scale=0.15,draw=symbols,line join=round]{#3}}%
	\expandafter\gdef\csname MH@symb@#1s\endcsname{\scalebox{0.75}{\tikzsetnextfilename{symbol#1}%
	\tikz[baseline=#2,scale=0.15,draw=symbols,line join=round]{#3}}}%
	\expandafter\gdef\csname MH@symb@#1ss\endcsname{\scalebox{0.65}{\tikzsetnextfilename{symbol#1}%
	\tikz[baseline=#2,scale=0.15,draw=symbols,line join=round]{#3}}}%
	}
\def\<#1>{\ifthenelse{\boolean{mmode}}{\mathchoice{\csname MH@symb@#1\endcsname}{\csname MH@symb@#1\endcsname}{\csname MH@symb@#1s\endcsname}{\csname MH@symb@#1ss\endcsname}}{\csname MH@symb@#1\endcsname}}
\makeatother

\DeclareSymbol{Xi22}{0.5}{\draw (0,0) node[xi] {} -- (-1,1) node[not] {} -- (0,2) node[xi] {};} 
\DeclareSymbol{Xi2}{-2}{\draw (-1,-0.25) node[xi] {} -- (0,1) node[xi] {};} 
\DeclareSymbol{Xi2b}{-2}{\draw (-1,-0.25) node[xic] {} -- (0,1) node[xic] {};} 
\DeclareSymbol{Xi2g}{-2}{\draw (-1,-0.25) node[xies] {} -- (0,1) node[xi] {};} 
\DeclareSymbol{Xi2g2}{-2}{\draw (-1,-0.25) node[xi] {} -- (0,1) node[xies] {};} 
\DeclareSymbol{cXi2}{-2}{\draw (0,-0.25) node[xi] {} -- (-1,1) node[xic] {};}
\DeclareSymbol{Xi3}{0}{\draw (0,0) node[xi] {} -- (-1,1) node[xi] {} -- (0,2) node[xi] {};}
\DeclareSymbol{XiIIXi}{0}{\draw (0,0) node[xi] {} -- (-1,1); \draw[kernels2] (-1,1) node[not] {} -- (0,2) node[xi] {};}

\DeclareSymbol{Xi4}{2}{\draw (0,0) node[xi] {} -- (-1,1) node[xi] {} -- (0,2) node[xi] {} -- (-1,3) node[xi] {};}
\DeclareSymbol{Xi4_1}{2}{\draw (0,0) node[xic] {} -- (-1,1) node[xic] {} -- (0,2) node[xi] {} -- (-1,3) node[xi] {};}
\DeclareSymbol{Xi4_2}{2}{\draw (0,0) node[xic] {} -- (-1,1) node[xi] {} -- (0,2) node[xi] {} -- (-1,3) node[xic] {};}
\DeclareSymbol{Xi2X}{-2}{\draw (0,-0.25) node[xi] {} -- (-1,1) node[xix] {};}
\DeclareSymbol{XXi2}{-2}{\draw (0,-0.25) node[xix] {} -- (-1,1) node[xi] {};}
\DeclareSymbol{IIXi}{0}{\draw (0,-0.25) node[not] {} -- (-1,1) node[xi] {} -- (0,2) node[xi] {};}
\DeclareSymbol{IXi^2}{-1}{\draw (-1,1) node[xi] {} -- (0,0) node[not] {} -- (1,1) node[xi] {};}
\DeclareSymbol{IIXi^2}{-4}{\draw (0,-1.5) node[not] {} -- (0,0);
\draw[kernels2] (-1,1) node[xi] {} -- (0,0) node[not] {} -- (1,1) node[xi] {};}
\DeclareSymbol{XiX}{-2.8}{\node[xibx] {};}
\DeclareSymbol{tauX}{-2.8}{ \node[X] {};}
\DeclareSymbol{Xi}{-2.8}{\node[xib] {};}

\DeclareSymbol{IXiX}{-1}{\draw (0,-0.25) node[not] {} -- (-1,1) node[xix] {};}
\DeclareSymbol{IXi3}{2}{\draw (0,-0.25) node[not] {} -- (-1,1) node[xi] {} -- (0,2) node[xi] {} -- (-1,3) node[xi] {};}
\DeclareSymbol{IXi}{-2}{\draw (0,-0.25) node[not] {} -- (-1,1) node[xi] {};}
\DeclareSymbol{XiI}{-2}{\draw (0,-0.25) node[xi] {} -- (-1,1) node[not] {};}

\DeclareSymbol{Xi4b}{0}{\draw(0,1.5) node[xi] {} -- (0,0); \draw (-1,1) node[xi] {} -- (0,0) node[xi] {} -- (1,1) node[xi] {};}
\DeclareSymbol{Xi4b'}{0}{\draw(0,1.5) node[xi] {} -- (0,-0.2); \draw (-1,1) node[xi] {} -- (0,-0.2) node[not] {} -- (1,1) node[xi] {};}
\DeclareSymbol{Xi4c}{0}{\draw (0,1) -- (0.8,2.2) node[xi] {};\draw (0,-0.25) node[xi] {} -- (0,1) node[xi] {} -- (-0.8,2.2) node[xi] {};}
\DeclareSymbol{Xi4d}{-4.5}{\draw (0,-1.5) node[not] {} -- (0,0); \draw (-1,1) node[xi] {} -- (0,0) node[xi] {} -- (1,1) node[xi] {};}
\DeclareSymbol{Xi4e}{0}{\draw (0,2) node[xi] {} -- (-1,1) node[xi] {} -- (0,0) node[xi] {} -- (1,1) node[xi] {};}
\DeclareSymbol{Xi4e'}{0}{\draw (0,2) node[xi] {} -- (-1,1) node[xi] {} -- (0,-0.2) node[not] {} -- (1,1) node[xi] {};}

\DeclareSymbol{Xitwo}
{0}{\draw[kernels2] (0,0) node[not] {} -- (-1,1) node[not] {}
-- (-2,2) node[not]{} -- (-3,3) node[xi]  {};
\draw[kernels2] (0,0) -- (1,1) node[xi] {};
\draw[kernels2] (-1,1) -- (0,2) node[xi] {};
\draw[kernels2] (-2,2) -- (-1,3) node[xi] {};}

\DeclareSymbol{IXitwo}
{0}{\draw (-.7,1.2) node[xi] {} -- (0,-0.2) -- (.7,1.2) node[xi] {};}
\DeclareSymbol{I1Xitwo}
{0}{\draw[kernels2] (0,0) node[not] {} -- (-1,1) node[xi] {};
\draw[kernels2] (0,0) -- (1,1) node[xi] {};}

\DeclareSymbol{I1Xitwobis}
{0}{\draw[kernels2] (0,0) node[not] {} -- (-1,1) node[xies] {};
\draw[kernels2] (0,0) -- (1,1) node[xies] {};}

\DeclareSymbol{I1Xitwog}
{0}{\draw[kernels2] (0,0) node[not] {} -- (-1,1) node[xies] {};
\draw[kernels2] (0,0) -- (1,1) node[xi] {};}

\DeclareSymbol{cI1Xitwo}
{0}{\draw[kernels2] (0,0) node[not] {} -- (-1,1) node[xic] {};
\draw[kernels2] (0,0) -- (1,1) node[xi] {};}

\DeclareSymbol{I1IXi3}{0}{\draw (0,0) node[xi] {} -- (-1,1) ; 
\draw[kernels2] (-1,1) node[not] {} -- (0,2) node[xi] {};
\draw[kernels2] (-1,1) node[not] {} -- (-2,2) node[xi] {};}

\DeclareSymbol{I1Xi3c}{-1}{\draw[kernels2](0,1.5) node[xi] {} -- (0,0) node[not] {}; \draw (-1,1) node[xi] {} -- (0,0) ; \draw[kernels2] (0,0) -- (1,1) node[xi] {};}

\DeclareSymbol{I1Xi3cbis}{-1}{\draw[kernels2](0,1.5) node[xies] {} -- (0,0) node[not] {}; \draw (-1,1) node[xies] {} -- (0,0) ; \draw[kernels2] (0,0) -- (1,1) node[xies] {};}

\DeclareSymbol{I1IXi3b}{0}{\draw[kernels2] (0,0) node[not] {} -- (-1,1) ; \draw[kernels2] (0,0)   -- (1,1) node[xi] {} ;
\draw (-1,1) node[xi] {} -- (0,2) node[xi] {};
}

\DeclareSymbol{I1IXi3c}{0}{\draw[kernels2] (0,0) node[not] {} -- (-1,1) ; \draw[kernels2] (0,0)   -- (1,1) node[xi] {} ;
\draw[kernels2] (-1,1) node[not] {} -- (0,2) node[xi] {};
\draw[kernels2] (-1,1) node[not] {} -- (-2,2) node[xi] {};}

\DeclareSymbol{I1IXi3cbis}{0}{\draw[kernels2] (0,0) node[not] {} -- (-1,1) ; \draw[kernels2] (0,0)   -- (1,1) node[xies] {} ;
\draw[kernels2] (-1,1) node[not] {} -- (0,2) node[xies] {};
\draw[kernels2] (-1,1) node[not] {} -- (-2,2) node[xies] {};}

\DeclareSymbol{I1Xi}{0}{\draw[kernels2] (0,0) node[not] {} -- (-1,1)  node[xi] {} ;}

\DeclareSymbol{I1Xi4a}{2}{\draw[kernels2] (0,0) node[not] {} -- (-1,1) ; \draw[kernels2] (0,0) node[not] {} -- (1,1) node[xi] {} ;
\draw (-1,1) node[xi] {} -- (0,2) node[xi] {} -- (-1,3) node[xi] {};}

\DeclareSymbol{cI1Xi4a}{2}{\draw[kernels2] (0,0) node[not] {} -- (-1,1) ; \draw[kernels2] (0,0) node[not] {} -- (1,1) node[xic] {} ;
\draw (-1,1) node[xic] {} -- (0,2) node[xi] {} -- (-1,3) node[xi] {};}

\DeclareSymbol{I1Xi4b}{2}{\draw (0,0) node[xi] {} -- (-1,1) node[xi] {} -- (0,2) ; \draw[kernels2] (0,2) node[not] {} -- (-1,3) node[xi] {};\draw[kernels2] (0,2)  -- (1,3) node[xi] {};
}

\DeclareSymbol{cI1Xi4b}{2}{\draw (0,0) node[xic] {} -- (-1,1) node[xic] {} -- (0,2) ; \draw[kernels2] (0,2) node[not] {} -- (-1,3) node[xi] {};\draw[kernels2] (0,2)  -- (1,3) node[xi] {};
}

\DeclareSymbol{I1Xi4c}{2}{\draw (0,0) node[xi] {} -- (-1,1) node[not] {}; \draw[kernels2] (-1,1) -- (0,2) ; 
\draw[kernels2] (-1,1) -- (-2,2) node[xi] {} ;
\draw (0,2) node[xi] {} -- (-1,3) node[xi] {};}

\DeclareSymbol{cI1Xi4c}{2}{\draw (0,0) node[xic] {} -- (-1,1) node[not] {}; \draw[kernels2] (-1,1) -- (0,2) ; 
\draw[kernels2] (-1,1) -- (-2,2) node[xic] {} ;
\draw (0,2) node[xi] {} -- (-1,3) node[xi] {};}

\DeclareSymbol{I1Xi4ab}{2}{\draw[kernels2] (0,0) node[not] {} -- (-1,1) ; \draw[kernels2] (0,0) node[not] {} -- (1,1) node[xi] {};\draw (-1,1) node[xi] {} -- (0,2) ; \draw[kernels2] (0,2) node[not] {} -- (-1,3) node[xi] {};\draw[kernels2] (0,2)  -- (1,3) node[xi] {}; }

\DeclareSymbol{cI1Xi4ab}{2}{\draw[kernels2] (0,0) node[not] {} -- (-1,1) ; \draw[kernels2] (0,0) node[not] {} -- (1,1) node[xic] {};\draw (-1,1) node[xic] {} -- (0,2) ; \draw[kernels2] (0,2) node[not] {} -- (-1,3) node[xi] {};\draw[kernels2] (0,2)  -- (1,3) node[xi] {}; }

\DeclareSymbol{I1Xi4bc}{2}{\draw (0,0) node[xi] {} -- (-1,1) node[not] {}; \draw[kernels2] (-1,1) -- (0,2) ; 
\draw[kernels2] (-1,1) -- (-2,2) node[xi] {} ; \draw[kernels2] (0,2) node[not] {} -- (-1,3) node[xi] {};\draw[kernels2] (0,2)  -- (1,3) node[xi] {};
}

\DeclareSymbol{cI1Xi4bc}{2}{\draw (0,0) node[xic] {} -- (-1,1) node[not] {}; \draw[kernels2] (-1,1) -- (0,2) ; 
\draw[kernels2] (-1,1) -- (-2,2) node[xic] {} ; \draw[kernels2] (0,2) node[not] {} -- (-1,3) node[xi] {};\draw[kernels2] (0,2)  -- (1,3) node[xi] {};
}

\DeclareSymbol{I1Xi4abcc1}{2}{\draw[kernels2] (0,0) node[not] {} -- (-1,1) node[not] {}
-- (-2,2) node[not]{} -- (-3,3) node[xic]  {};
\draw[kernels2] (0,0) -- (1,1) node[xic] {};
\draw[kernels2] (-1,1) -- (0,2) node[xi] {};
\draw[kernels2] (-2,2) -- (-1,3) node[xi] {};
}

\DeclareSymbol{I1Xi4abcc1b}{2}{\draw[kernels2] (0,0) node[not] {} -- (-1,1) node[not] {}
-- (-2,2) node[not]{} -- (-3,3) node[xi]  {};
\draw[kernels2] (0,0) -- (1,1) node[xic] {};
\draw[kernels2] (-1,1) -- (0,2) node[xic] {};
\draw[kernels2] (-2,2) -- (-1,3) node[xi] {};
}

\DeclareSymbol{I1Xi4abcc2}{2}{\draw[kernels2] (0,0) node[not] {} -- (-1,1) node[not] {}
-- (-2,2) node[not]{} -- (-3,3) node[xic]  {};
\draw[kernels2] (0,0) -- (1,1) node[xi] {};
\draw[kernels2] (-1,1) -- (0,2) node[xi] {};
\draw[kernels2] (-2,2) -- (-1,3) node[xic] {};
}

\DeclareSymbol{I1Xi4ac}{2}{\draw[kernels2] (0,0) node[not] {} -- (-1,1) ; \draw[kernels2] (0,0) node[not] {} -- (1,1) node[xi] {}; 
\draw[kernels2] (-1,1) node[not] {} -- (0,2) ; 
\draw[kernels2] (-1,1) -- (-2,2) node[xi] {} ;
\draw (0,2) node[xi] {} -- (-1,3) node[xi] {};}

\DeclareSymbol{cI1Xi4ac}{2}{\draw[kernels2] (0,0) node[not] {} -- (-1,1) ; \draw[kernels2] (0,0) node[not] {} -- (1,1) node[xic] {}; 
\draw[kernels2] (-1,1) node[not] {} -- (0,2) ; 
\draw[kernels2] (-1,1) -- (-2,2) node[xic] {} ;
\draw (0,2) node[xi] {} -- (-1,3) node[xi] {};}

\DeclareSymbol{I1Xi4acc1}{2}{\draw[kernels2] (0,0) node[not] {} -- (-1,1) ; \draw[kernels2] (0,0) node[not] {} -- (1,1) node[xic] {}; 
\draw[kernels2] (-1,1) node[not] {} -- (0,2) ; 
\draw[kernels2] (-1,1) -- (-2,2) node[xi] {} ;
\draw (0,2) node[xic] {} -- (-1,3) node[xi] {};}

\DeclareSymbol{I1Xi4acc2}{2}{\draw[kernels2] (0,0) node[not] {} -- (-1,1) ; \draw[kernels2] (0,0) node[not] {} -- (1,1) node[xic] {}; 
\draw[kernels2] (-1,1) node[not] {} -- (0,2) ; 
\draw[kernels2] (-1,1) -- (-2,2) node[xi] {} ;
\draw (0,2) node[xi] {} -- (-1,3) node[xic] {};}

\DeclareSymbol{2I1Xi4}{2}{\draw[kernels2] (0,0) node[not] {} -- (-1,1) node[not] {};
\draw[kernels2] (0,0) -- (1,1) node[not] {};
\draw[kernels2] (-1,1) -- (-1.5,2.5) node[xi] {};
\draw[kernels2] (-1,1) -- (-0.5,2.5) node[xi] {};
\draw[kernels2] (1,1) -- (0.5,2.5) node[xi] {};
\draw[kernels2] (1,1) -- (1.5,2.5) node[xi] {};
}

\DeclareSymbol{2I1Xi4dis}{2}{\draw[kernels2] (0,0) node[not] {} -- (-1,1) node[not] {};
\draw[kernels2] (0,0) -- (1,1) node[not] {};
\draw[kernels2] (-1,1) -- (-1.5,2.5) node[xies] {};
\draw[kernels2] (-1,1) -- (-0.5,2.5) node[xies] {};
\draw[kernels2] (1,1) -- (0.5,2.5) node[xies] {};
\draw[kernels2] (1,1) -- (1.5,2.5) node[xies] {};
}

\DeclareSymbol{2I1Xi4c1}{2}{\draw[kernels2] (0,0) node[not] {} -- (-1,1) node[not] {};
\draw[kernels2] (0,0) -- (1,1) node[not] {};
\draw[kernels2] (-1,1) -- (-1.5,2.5) node[xic] {};
\draw[kernels2] (-1,1) -- (-0.5,2.5) node[xi] {};
\draw[kernels2] (1,1) -- (0.5,2.5) node[xic] {};
\draw[kernels2] (1,1) -- (1.5,2.5) node[xi] {};
}

\DeclareSymbol{2I1Xi4c2}{2}{\draw[kernels2] (0,0) node[not] {} -- (-1,1) node[not] {};
\draw[kernels2] (0,0) -- (1,1) node[not] {};
\draw[kernels2] (-1,1) -- (-1.5,2.5) node[xic] {};
\draw[kernels2] (-1,1) -- (-0.5,2.5) node[xic] {};
\draw[kernels2] (1,1) -- (0.5,2.5) node[xi] {};
\draw[kernels2] (1,1) -- (1.5,2.5) node[xi] {};
}

\DeclareSymbol{2I1Xi4b}{2}{\draw[kernels2] (0,0) node[not] {} -- (-1,1) ;
\draw[kernels2] (0,0) -- (1,1);
\draw (-1,1) node[xi] {} -- (-1,2.5) node[xi] {};
\draw (1,1)  node[xi] {} -- (1,2.5) node[xi] {};
}

\DeclareSymbol{2I1Xi4bb}{2}{\draw[kernels2] (0,0) node[not] {} -- (-1,1) ;
\draw[kernels2] (0,0) -- (1,1);
\draw (-1,1) node[xi] {} -- (-1,2.5) node[xiesf] {};
\draw (1,1)  node[xi] {} -- (1,2.5) node[xic] {};
}

\DeclareSymbol{2I1Xi4c}{2}{\draw[kernels2] (0,0) node[not] {} -- (-1,1);
\draw[kernels2] (0,0) -- (1,1) node[not] {};
\draw (-1,1)  node[xi] {} -- (-1,2.5) node[xi] {};
\draw[kernels2] (1,1) -- (0.4,2.5) node[xi] {};
\draw[kernels2] (1,1) -- (1.6,2.5) node[xi] {};
}

\DeclareSymbol{2I1Xi4cc1}{2}{\draw[kernels2] (0,0) node[not] {} -- (-1,1);
\draw[kernels2] (0,0) -- (1,1) node[not] {};
\draw (-1,1)  node[xic] {} -- (-1,2.5) node[xi] {};
\draw[kernels2] (1,1) -- (0.4,2.5) node[xic] {};
\draw[kernels2] (1,1) -- (1.6,2.5) node[xi] {};
}

\DeclareSymbol{2I1Xi4cc2}{2}{\draw[kernels2] (0,0) node[not] {} -- (-1,1);
\draw[kernels2] (0,0) -- (1,1) node[not] {};
\draw (-1,1)  node[xic] {} -- (-1,2.5) node[xic] {};
\draw[kernels2] (1,1) -- (0.4,2.5) node[xi] {};
\draw[kernels2] (1,1) -- (1.6,2.5) node[xi] {};
}

\DeclareSymbol{Xi4ba}{0}{\draw(-0.5,1.5) node[xi] {} -- (0,0); \draw (-1.5,1) node[xi] {} -- (0,0) node[not] {}; \draw[kernels2] (0,0) -- (1.5,1) node[xi] {};
\draw[kernels2] (0,0) -- (0.5,1.5) node[xi] {} ;}

\DeclareSymbol{Xi4badis}{0}{\draw(-0.5,1.5) node[xies] {} -- (0,0); \draw (-1.5,1) node[xies] {} -- (0,0) node[not] {}; \draw[kernels2] (0,0) -- (1.5,1) node[xies] {};
\draw[kernels2] (0,0) -- (0.5,1.5) node[xies] {} ;}

\DeclareSymbol{Xi4ba1}{0}{\draw(-0.5,1.5) node[xi] {} -- (0,0); \draw (-1.5,1) node[xi] {} -- (0,0) node[not] {}; \draw[kernels2] (0,0) -- (1.5,1) node[xic] {};
\draw[kernels2] (0,0) -- (0.5,1.5) node[xic] {} ;}

\DeclareSymbol{Xi4ba1b}{0}{\draw(-0.5,1.5) node[xic] {} -- (0,0); \draw (-1.5,1) node[xic] {} -- (0,0) node[not] {}; \draw[kernels2] (0,0) -- (1.5,1) node[xi] {};
\draw[kernels2] (0,0) -- (0.5,1.5) node[xi] {} ;}

\DeclareSymbol{Xi4ba1bdiff}{0}{\draw(-0.5,1.5) node[xic] {} -- (0,0); \draw (-1.5,1) node[xic] {} -- (0,0) node[not] {}; \draw (0,0) -- (1.5,1) node[xi] {};
\draw (0,0) -- (0.5,1.5) node[xi] {};
\draw(0,0) node[diff] {};}

\DeclareSymbol{Xi4ba1bb}{0}{\draw(-0.5,1.5) node[xic] {} -- (0,0); \draw (-1.5,1) node[xiesf] {} -- (0,0) node[not] {}; \draw[kernels2] (0,0) -- (1.5,1) node[xi] {};
\draw[kernels2] (0,0) -- (0.5,1.5) node[xi] {} ;}

\DeclareSymbol{Xi4ba2}{0}{\draw(-0.5,1.5) node[xi] {} -- (0,0); \draw (-1.5,1) node[xic] {} -- (0,0) node[not] {}; \draw[kernels2] (0,0) -- (1.5,1) node[xi] {};
\draw[kernels2] (0,0) -- (0.5,1.5) node[xic] {} ;}

\DeclareSymbol{Xi4ba2b}{0}{\draw(-0.5,1.5) node[xi] {} -- (0,0); \draw (-1.5,1) node[xic] {} -- (0,0) node[not] {}; \draw[kernels2] (0,0) -- (1.5,1) node[xi] {};
\draw[kernels2] (0,0) -- (0.5,1.5) node[xiesf] {} ;}

\DeclareSymbol{Xi4ca}{0}{\draw (0,1) -- (-1,2.2) node[xi] {};\draw (0,-0.25) node[xi] {} -- (0,1) ; \draw[kernels2] (0,1) node[not] {} -- (1,2.2) node[xi] {};
\draw[kernels2] (0,1) {} -- (0,2.7) node[xi] {};
}

\DeclareSymbol{Xi4cb}{0}{\draw (-1,1) -- (-2,2) node[xi] {};\draw[kernels2] (0,0)  -- (-1,1) node[xi] {} ; \draw[kernels2] (0,0) node[not] {} -- (1,1) node[xi] {} ; 
\draw (-1,1) node[xi] {} -- (0,2) node[xi] {};}

\DeclareSymbol{Xi4cbb}{0}{\draw (-1,1) -- (-2,2) node[xiesf] {};\draw[kernels2] (0,0)  -- (-1,1) node[xi] {} ; \draw[kernels2] (0,0) node[not] {} -- (1,1) node[xi] {} ; 
\draw (-1,1) node[xi] {} -- (0,2) node[xic] {};}

\DeclareSymbol{Xi4cbc1}{0}{\draw (-1,1) -- (-2,2) node[xic] {};\draw[kernels2] (0,0)  -- (-1,1) node[xic] {} ; \draw[kernels2] (0,0) node[not] {} -- (1,1) node[xi] {} ; 
\draw (-1,1) node[xic] {} -- (0,2) node[xi] {};}

\DeclareSymbol{Xi4cbc2}{0}{\draw (-1,1) -- (-2,2) node[xi] {};\draw[kernels2] (0,0)  -- (-1,1) node[xi] {} ; \draw[kernels2] (0,0) node[not] {} -- (1,1) node[xic] {} ; 
\draw (-1,1) node[xic] {} -- (0,2) node[xi] {};}

\DeclareSymbol{Xi4cab}{0}{\draw (-1,1) -- (-2,2) node[xi] {};\draw[kernels2] (0,0)  -- (-1,1); \draw[kernels2] (0,0) node[not] {} -- (1,1) node[xi] {} ; 
\draw[kernels2] (-1,1)  {} -- (0,2) node[xi] {};
\draw[kernels2] (-1,1) node[not] {} -- (-1,2.5) node[xi] {};
}

\DeclareSymbol{Xi4cabdis}{0}{\draw (-1,1) -- (-2,2) node[xies] {};\draw[kernels2] (0,0)  -- (-1,1); \draw[kernels2] (0,0) node[not] {} -- (1,1) node[xies] {} ; 
\draw[kernels2] (-1,1)  {} -- (0,2) node[xies] {};
\draw[kernels2] (-1,1) node[not] {} -- (-1,2.5) node[xies] {};
}

\DeclareSymbol{Xi4cabc1}{0}{\draw (-1,1) -- (-2,2) node[xi] {};\draw[kernels2] (0,0)  -- (-1,1); \draw[kernels2] (0,0) node[not] {} -- (1,1) node[xic] {} ; 
\draw[kernels2] (-1,1)  {} -- (0,2) node[xic] {};
\draw[kernels2] (-1,1) node[not] {} -- (-1,2.5) node[xi] {};
}

\DeclareSymbol{Xi4cabc2}{0}{\draw (-1,1) -- (-2,2) node[xic] {};\draw[kernels2] (0,0)  -- (-1,1); \draw[kernels2] (0,0) node[not] {} -- (1,1) node[xic] {} ; 
\draw[kernels2] (-1,1)  {} -- (0,2) node[xi] {};
\draw[kernels2] (-1,1) node[not] {} -- (-1,2.5) node[xi] {};
}

\DeclareSymbol{Xi4ea}{1.5}{\draw (-1,2.5) node[xi] {} -- (-1,1) node[xi] {} -- (0,0); 
 \draw[kernels2] (0,0)  -- (1,1) node[xi] {};
\draw[kernels2] (0,0) node[not] {} -- (0,1.5) node[xi] {}; }

\DeclareSymbol{Xi4eac1}{1.5}{\draw (-1,2.5) node[xic] {} -- (-1,1) node[xi] {} -- (0,0); 
 \draw[kernels2] (0,0)  -- (1,1) node[xic] {};
\draw[kernels2] (0,0) node[not] {} -- (0,1.5) node[xi] {}; }

\DeclareSymbol{Xi4eac1b}{1.5}{\draw (-1,2.5) node[xic] {} -- (-1,1) node[xi] {} -- (0,0); 
 \draw[kernels2] (0,0)  -- (1,1) node[xiesf] {};
\draw[kernels2] (0,0) node[not] {} -- (0,1.5) node[xi] {}; }

\DeclareSymbol{Xi4eac2}{1.5}{\draw (-1,2.5) node[xic] {} -- (-1,1) node[xic] {} -- (0,0); 
 \draw[kernels2] (0,0)  -- (1,1) node[xi] {};
\draw[kernels2] (0,0) node[not] {} -- (0,1.5) node[xi] {}; }

\DeclareSymbol{Xi4eact1}{1.5}{\draw (-1,2.5) node[xic] {} -- (-1,1) node[xi] {} -- (0,0); 
 \draw (0,0)  -- (1,1) node[xic] {};
\draw[rho] (0,0) node[not] {} -- (0,1.5) node[xi] {}; }

\DeclareSymbol{Xi4eact2}{1.5}{\draw[rho] (-1,2.5) node[xic] {} -- (-1,1) node[xi] {} -- (0,0); 
 \draw (0,0)  -- (1,1) node[xic] {};
\draw (0,0) node[not] {} -- (0,1.5) node[xi] {}; }

\DeclareSymbol{Xi4eabis}{1.5}{\draw (-1,2.5) node[xi] {} -- (-1,1) ; \draw[kernels2] (-1,1) node[xi] {} -- (0,0); 
 \draw (0,0)  -- (1,1) node[xi] {};
\draw[kernels2] (0,0) node[not] {} -- (0,1.5) node[xi] {}; }

\DeclareSymbol{Xi4eabisc1}{1.5}{\draw (-1,2.5) node[xic] {} -- (-1,1) ; \draw[kernels2] (-1,1) node[xi] {} -- (0,0); 
 \draw (0,0)  -- (1,1) node[xi] {};
\draw[kernels2] (0,0) node[not] {} -- (0,1.5) node[xic] {}; }

\DeclareSymbol{Xi4eabisc1b}{1.5}{\draw (-1,2.5) node[xic] {} -- (-1,1) ; \draw[kernels2] (-1,1) node[xi] {} -- (0,0); 
 \draw (0,0)  -- (1,1) node[xi] {};
\draw[kernels2] (0,0) node[not] {} -- (0,1.5) node[xiesf] {}; }

\DeclareSymbol{Xi4eabisc1bis}{1.5}{\draw (-1,2.5) node[xi] {} -- (-1,1) ; \draw[kernels2] (-1,1) node[xi] {} -- (0,0); 
 \draw (0,0)  -- (1,1) node[xi] {};
\draw[kernels2] (0,0) node[not] {} -- (0,1.5) node[xi] {};
\draw (-2,1) node[] {\tiny{$i$}};
\draw (-2,2.5) node[] {\tiny{$\ell$}};
\draw (2,1) node[] {\tiny{$k$}};
\draw (0,2.5) node[] {\tiny{$j$}};
 }

\DeclareSymbol{Xi4eabisc1tris}{1.5}{\draw (-1,2.5) node[xi] {} -- (-1,1) ; \draw[kernels2] (-1,1) node[xi] {} -- (0,0); 
 \draw (0,0)  -- (1,1) node[xi] {};
\draw[kernels2] (0,0) node[not] {} -- (0,1.5) node[xi] {};
\draw (-2,1) node[] {\tiny{i}};
\draw (-2,2.5) node[] {\tiny{j}};
\draw (2,1) node[] {\tiny{j}};
\draw (0,2.5) node[] {\tiny{i}};
 }

\DeclareSymbol{Xi4eabisc1quater}{1.5}{\draw (-1,2.5) node[xic] {} -- (-1,1) ; \draw[kernels2] (-1,1) node[xi] {} -- (0,0); 
 \draw (0,0)  -- (1,1) node[xic] {};
\draw[kernels2] (0,0) node[not] {} -- (0,1.5) node[xi] {};
 }

\DeclareSymbol{Xi4eabisc2}{1.5}{\draw (-1,2.5) node[xic] {} -- (-1,1) ; \draw[kernels2] (-1,1) node[xi] {} -- (0,0); 
 \draw (0,0)  -- (1,1) node[xic] {};
\draw[kernels2] (0,0) node[not] {} -- (0,1.5) node[xi] {}; }

\DeclareSymbol{Xi4eabisc2l}{1.5}{\draw (-1,2.5) node[xiesf] {} -- (-1,1) ; \draw[kernels2] (-1,1) node[xi] {} -- (0,0); 
 \draw (0,0)  -- (1,1) node[xic] {};
\draw[kernels2] (0,0) node[not] {} -- (0,1.5) node[xi] {}; }

\DeclareSymbol{Xi4eabisc2r}{1.5}{\draw (-1,2.5) node[xic] {} -- (-1,1) ; \draw[kernels2] (-1,1) node[xi] {} -- (0,0); 
 \draw (0,0)  -- (1,1) node[xiesf] {};
\draw[kernels2] (0,0) node[not] {} -- (0,1.5) node[xi] {}; }

\DeclareSymbol{Xi4eabisc3}{1.5}{\draw (-1,2.5) node[xic] {} -- (-1,1) ; \draw[kernels2] (-1,1) node[xic] {} -- (0,0); 
 \draw (0,0)  -- (1,1) node[xi] {};
\draw[kernels2] (0,0) node[not] {} -- (0,1.5) node[xi] {}; }

\DeclareSymbol{Xi4eb}{0}{
\draw[kernels2] (0,2) node[xi] {} -- (-1,1) ; \draw[kernels2] (-2,2)  node[xi] {} -- (-1,1) ; \draw (-1,1)  node[not] {} -- (0,0); 
 \draw (0,0) node[xi] {}  -- (1,1) node[xi] {};
}

\DeclareSymbol{Xi4eab}{1.5}{\draw[kernels2] (-1,2.5) node[xi] {} -- (-1,1) ; \draw[kernels2] (-2,2)  node[xi] {} -- (-1,1) ; \draw (-1,1)  node[not] {} -- (0,0); 
 \draw[kernels2] (0,0)  -- (1,1) node[xi] {};
\draw[kernels2] (0,0) node[not] {} -- (0,1.5) node[xi] {}; 
}

\DeclareSymbol{Xi4eabdis}{1.5}{\draw[kernels2] (-1,2.5) node[xies] {} -- (-1,1) ; \draw[kernels2] (-2,2)  node[xies] {} -- (-1,1) ; \draw (-1,1)  node[not] {} -- (0,0); 
 \draw[kernels2] (0,0)  -- (1,1) node[xies] {};
\draw[kernels2] (0,0) node[not] {} -- (0,1.5) node[xies] {}; 
}

\DeclareSymbol{Xi4eabc1}{1.5}{\draw[kernels2] (-1,2.5) node[xic] {} -- (-1,1) ; \draw[kernels2] (-2,2)  node[xi] {} -- (-1,1) ; \draw (-1,1)  node[not] {} -- (0,0); 
 \draw[kernels2] (0,0)  -- (1,1) node[xic] {};
\draw[kernels2] (0,0) node[not] {} -- (0,1.5) node[xi] {}; 
}

\DeclareSymbol{Xi4eabc2}{1.5}{\draw[kernels2] (-1,2.5) node[xi] {} -- (-1,1) ; \draw[kernels2] (-2,2)  node[xi] {} -- (-1,1) ; \draw (-1,1)  node[not] {} -- (0,0); 
 \draw[kernels2] (0,0)  -- (1,1) node[xic] {};
\draw[kernels2] (0,0) node[not] {} -- (0,1.5) node[xic] {}; 
}

\DeclareSymbol{Xi4eabbis}{1.5}{\draw[kernels2] (-1,2.5) node[xi] {} -- (-1,1) ; \draw[kernels2] (-2,2)  node[xi] {} -- (-1,1) ; \draw[kernels2] (-1,1)  node[not] {} -- (0,0); 
 \draw (0,0)  -- (1,1) node[xi] {};
\draw[kernels2] (0,0) node[not] {} -- (0,1.5) node[xi] {}; 
}

\DeclareSymbol{Xi4eabbisc1}{1.5}{\draw[kernels2] (-1,2.5) node[xic] {} -- (-1,1) ; \draw[kernels2] (-2,2)  node[xi] {} -- (-1,1) ; \draw[kernels2] (-1,1)  node[not] {} -- (0,0); 
 \draw (0,0)  -- (1,1) node[xic] {};
\draw[kernels2] (0,0) node[not] {} -- (0,1.5) node[xi] {}; 
}

\DeclareSymbol{Xi4eabbisc1perm}{1.5}{\draw[kernels2] (-1,2.5) node[xi] {} -- (-1,1) ; \draw[kernels2] (-2,2)  node[xic] {} -- (-1,1) ; \draw[kernels2] (-1,1)  node[not] {} -- (0,0); 
 \draw (0,0)  -- (1,1) node[xic] {};
\draw[kernels2] (0,0) node[not] {} -- (0,1.5) node[xi] {}; 
}

\DeclareSymbol{Xi4eabbisc2}{1.5}{\draw[kernels2] (-1,2.5) node[xi] {} -- (-1,1) ; \draw[kernels2] (-2,2)  node[xi] {} -- (-1,1) ; \draw[kernels2] (-1,1)  node[not] {} -- (0,0); 
 \draw (0,0)  -- (1,1) node[xic] {};
\draw[kernels2] (0,0) node[not] {} -- (0,1.5) node[xic] {}; 
}

\DeclareSymbol{Xi2cbis}{0}{\draw[kernels2] (0,1) -- (0.8,2.2) node[xi] {};\draw[kernels2] (0,-0.25) node[not] {} -- (0,1); \draw[kernels2] (0,1) node[not] {} -- (-0.8,2.2) node[xi] {};}

\DeclareSymbol{Xi2cbis1}{0}{\draw (0,1) -- (-0.8,2.2) node[xi] {};\draw[kernels2] (0,-0.25) node[not] {} -- (0,1) node[xi] {}; }

\DeclareSymbol{Xi2Xbis}{-2}{\draw[kernels2] (0,-0.25)  -- (-1,1) ; \draw (-1,1) node[xix] {};
\draw[kernels2] (0,-0.25) node[not] {} -- (1,1) node[xi] {};}

\DeclareSymbol{XXi2bis}{-2}{\draw[kernels2] (0,-0.25) -- (-1,1) node[xi] {};
\draw[kernels2] (0,-0.25) node[X] {} -- (1,1) node[xi] {};}

\DeclareSymbol{I1XiIXi}{0}{\draw[kernels2] (0,-0.25) -- (1,1) node[xi] {};
\draw (0,-0.25) node[not] {} -- (-1,1) node[xi] {};}

\DeclareSymbol{I1XiIXib}{0}{\draw  (0,-0.25) node[xi] {} -- (0,1) node[not] {};
\draw[kernels2] (0,1) -- (0,2.25) ; \draw (0,2.25) node[xi]{}; }

\DeclareSymbol{I1XiIXic}{0}{
\draw[kernels2] (0,0) -- (1,1) node[xi] {} ; 
\draw[kernels2] (0,0) node[not] {}  -- (-1,1) node[not] {} -- (0,2) node[xi] {};
}

\DeclareSymbol{thin}{1.4}{\draw[pagebackground] (-0.3,0) -- (0.3,0); \draw  (0,0) -- (0,2);}
\DeclareSymbol{thin2}{1.4}{\draw[pagebackground] (-0.3,0) -- (0.3,0); \draw[tinydots]  (0,0) -- (0,2);}

\DeclareSymbol{thick}{1.4}{\draw[pagebackground] (-0.3,0) -- (0.3,0); \draw[kernels2]  (0,0) -- (0,2);}

\DeclareSymbol{thick2}{1.4}{\draw[pagebackground] (-0.3,0) -- (0.3,0); \draw[kernels2,tinydots]  (0,0) -- (0,2);}

\DeclareSymbol{Xi4ind}{2}{\draw (0,0) node[xi,label={[label distance=-0.2em]right: \scriptsize  $ i $}]  { } -- (-1,1) node[xi,label={[label distance=-0.2em]left: \scriptsize  $ j $}] {} -- (0,2) node[xi,label={[label distance=-0.2em]right: \scriptsize  $ k $}] {} -- (-1,3) node[xi,label={[label distance=-0.2em]left: \scriptsize  $ \ell $}] {};}

\DeclareSymbol{Xi4c1}{2}{\draw (0,0) node[xic] {} -- (-1,1) node[xi] {} -- (0,2) node[xic] {} -- (-1,3) node[xi] {};} 
\DeclareSymbol{IXi2ex}{0}{\draw (0,-0.25) node[xie] {} -- (-1,1) node[xi] {} ; \draw (0,-0.25)-- (1,1) node[xi] {};}
\DeclareSymbol{IXi2ex1}{0}{\draw (0,-0.25) node[xie] {} -- (-1,1) node[xi] {} -- (0,2) node[xi] {};}

\DeclareSymbol{Xi4b1}{0}{\draw(0,1.5) node[xic] {} -- (0,0); \draw (-1,1) node[xic] {} -- (0,0) node[xi] {} -- (1,1) node[xi] {};}

\DeclareSymbol{Xi4ec1}{0}{\draw (0,2) node[xi] {} -- (-1,1) node[xic] {} -- (0,0) node[xic] {} -- (1,1) node[xi] {};}
\DeclareSymbol{Xi4ec2}{0}{\draw (0,2) node[xic] {} -- (-1,1) node[xi] {} -- (0,0) node[xic] {} -- (1,1) node[xi] {};}
\DeclareSymbol{Xi4ec3}{0}{\draw (0,2) node[xic] {} -- (-1,1) node[xic] {} -- (0,0) node[xi] {} -- (1,1) node[xi] {};}

\DeclareSymbol{I1Xi4ac1}{2}{\draw[kernels2] (0,0) node[not] {} -- (-1,1) ; \draw[kernels2] (0,0) node[not] {} -- (1,1) node[xic] {} ;
\draw (-1,1) node[xi] {} -- (0,2) node[xic] {} -- (-1,3) node[xi] {};}

\DeclareSymbol{I1Xi4ac2}{2}{\draw[kernels2] (0,0) node[not] {} -- (-1,1) ; \draw[kernels2] (0,0) node[not] {} -- (1,1) node[xic] {} ;
\draw (-1,1) node[xi] {} -- (0,2) node[xi] {} -- (-1,3) node[xic] {};}

\DeclareSymbol{I1Xi4bp}{2}{\draw (0,0) node[not] {} -- (-1,1) node[xi] {} -- (0,2) ; \draw[kernels2] (0,2) node[not] {} -- (-1,3) node[xi] {};\draw[kernels2] (0,2)  -- (1,3) node[xi] {};
}

\DeclareSymbol{I1Xi4bc1}{2}{\draw (0,0) node[xic] {} -- (-1,1) node[xi] {} -- (0,2) ; \draw[kernels2] (0,2) node[not] {} -- (-1,3) node[xi] {};\draw[kernels2] (0,2)  -- (1,3) node[xic] {};
}

\DeclareSymbol{I1Xi4bc2}{2}{\draw (0,0) node[xic] {} -- (-1,1) node[xi] {} -- (0,2) ; \draw[kernels2] (0,2) node[not] {} -- (-1,3) node[xic] {};\draw[kernels2] (0,2)  -- (1,3) node[xi] {};
}

\DeclareSymbol{I1Xi4cp}{2}{\draw (0,0) node[not] {} -- (-1,1) node[not] {}; \draw[kernels2] (-1,1) -- (0,2) ; 
\draw[kernels2] (-1,1) -- (-2,2) node[xi] {} ;
\draw (0,2) node[xi] {} -- (-1,3) node[xi] {};}

\DeclareSymbol{I1Xi4cc1}{2}{\draw (0,0) node[xic] {} -- (-1,1) node[not] {}; \draw[kernels2] (-1,1) -- (0,2) ; 
\draw[kernels2] (-1,1) -- (-2,2) node[xi] {} ;
\draw (0,2) node[xic] {} -- (-1,3) node[xi] {};}

\DeclareSymbol{I1Xi4cc2}{2}{\draw (0,0) node[xic] {} -- (-1,1) node[not] {}; \draw[kernels2] (-1,1) -- (0,2) ; 
\draw[kernels2] (-1,1) -- (-2,2) node[xi] {} ;
\draw (0,2) node[xi] {} -- (-1,3) node[xic] {};}

\DeclareSymbol{I1Xi4abc1}{2}{\draw[kernels2] (0,0) node[not] {} -- (-1,1) ; \draw[kernels2] (0,0) node[not] {} -- (1,1) node[xic] {};\draw (-1,1) node[xi] {} -- (0,2) ; \draw[kernels2] (0,2) node[not] {} -- (-1,3) node[xic] {};\draw[kernels2] (0,2)  -- (1,3) node[xi] {}; }

\DeclareSymbol{I1Xi4abc2}{2}{\draw[kernels2] (0,0) node[not] {} -- (-1,1) ; \draw[kernels2] (0,0) node[not] {} -- (1,1) node[xic] {};\draw (-1,1) node[xi] {} -- (0,2) ; \draw[kernels2] (0,2) node[not] {} -- (-1,3) node[xi] {};\draw[kernels2] (0,2)  -- (1,3) node[xic] {}; }

\DeclareSymbol{R1}{0}{\draw (-1,1) node[xi] {} -- (0,0) node[not] {};
\draw[kernels2] (0,1.5) node[xic] {} -- (0,0) -- (1,1) node[xic] {};}
\DeclareSymbol{R2}{0}{\draw (-1,1) node[xic] {} -- (0,0) node[not] {};
\draw[kernels2] (0,1.5)  {} -- (0,0) -- (1,1)  {};
\draw (0,1.5) node[xi] {};
\draw (1,1) node[xic] {};
}
\DeclareSymbol{R3}{1}{\draw[kernels2] (-1,1.5)  {} -- (0,0) node[not] {} -- (1,1.5);
\draw (-1,1.5) node[xi] {};
\draw[kernels2] (0,3) {} -- (1,1.5) -- (2,3)  {};
\draw  (0,3) node[xic] {} ;
\draw (2,3) node[xic] {};}
\DeclareSymbol{R4}{1}{\draw[kernels2] (-1,1.5) node[xic] {} -- (0,0) node[not] {} -- (1,1.5);
\draw[kernels2] (0,3) {} -- (1,1.5) -- (2,3) node[xic] {};
\draw (0,3) node[xi] {};}

\DeclareSymbol{I1Xi4bcp}{2}{\draw (0,0) node[not] {} -- (-1,1) node[not] {}; \draw[kernels2] (-1,1) -- (0,2) ; 
\draw[kernels2] (-1,1) -- (-2,2) node[xi] {} ; \draw[kernels2] (0,2) node[not] {} -- (-1,3) node[xi] {};\draw[kernels2] (0,2)  -- (1,3) node[xi] {};
}

\DeclareSymbol{I1Xi4bcc1}{2}{\draw (0,0) node[xic] {} -- (-1,1) node[not] {}; \draw[kernels2] (-1,1) -- (0,2) ; 
\draw[kernels2] (-1,1) -- (-2,2) node[xi] {} ; \draw[kernels2] (0,2) node[not] {} -- (-1,3) node[xi] {};\draw[kernels2] (0,2)  -- (1,3) node[xic] {};
}

\DeclareSymbol{I1Xi4bcc2}{2}{\draw (0,0) node[xic] {} -- (-1,1) node[not] {}; \draw[kernels2] (-1,1) -- (0,2) ; 
\draw[kernels2] (-1,1) -- (-2,2) node[xi] {} ; \draw[kernels2] (0,2) node[not] {} -- (-1,3) node[xic] {};\draw[kernels2] (0,2)  -- (1,3) node[xi] {};
} 

\DeclareSymbol{2I1Xi4bc1}{2}{\draw[kernels2] (0,0) node[not] {} -- (-1,1) ;
\draw[kernels2] (0,0) -- (1,1);
\draw (-1,1) node[xic] {} -- (-1,2.5) node[xi] {};
\draw (1,1)  node[xic] {} -- (1,2.5) node[xi] {};
}

\DeclareSymbol{2I1Xi4bc2}{2}{\draw[kernels2] (0,0) node[not] {} -- (-1,1) ;
\draw[kernels2] (0,0) -- (1,1);
\draw (-1,1) node[xi] {} -- (-1,2.5) node[xic] {};
\draw (1,1)  node[xic] {} -- (1,2.5) node[xi] {};
}

\DeclareSymbol{diff2I1Xi4bc2}{2}{\draw (0,0) node[diff] {} -- (-1,1) ;
\draw (0,0) -- (1,1);
\draw (-1,1) node[xi] {} -- (-1,2.5) node[xic] {};
\draw (1,1)  node[xic] {} -- (1,2.5) node[xi] {};
}

\DeclareSymbol{2I1Xi4bc3}{2}{\draw[kernels2] (0,0) node[not] {} -- (-1,1) ;
\draw[kernels2] (0,0) -- (1,1);
\draw (-1,1) node[xic] {} -- (-1,2.5) node[xic] {};
\draw (1,1)  node[xi] {} -- (1,2.5) node[xi] {};
}

\DeclareSymbol{Xi41}{0}{\draw (0,1) -- (0.8,2.2) node[xic] {};\draw (0,-0.25) node[xi] {} -- (0,1) node[xi] {} -- (-0.8,2.2) node[xic] {};} 

\DeclareSymbol{Xi42}{0}{\draw (0,1) -- (0.8,2.2) node[xi] {};\draw (0,-0.25) node[xic] {} -- (0,1) node[xi] {} -- (-0.8,2.2) node[xic] {};}

\DeclareSymbol{Xi4ca1}{0}{\draw (0,1) -- (-1,2.2) node[xic] {};\draw (0,-0.25) node[xi] {} -- (0,1) ; \draw[kernels2] (0,1) node[not] {} -- (1,2.2) node[xic] {};
\draw[kernels2] (0,1) {} -- (0,2.7) node[xi] {};
}

\DeclareSymbol{Xi4ca2}{0}{\draw (0,1) -- (-1,2.2) node[xi] {};\draw (0,-0.25) node[xi] {} -- (0,1) ; \draw[kernels2] (0,1) node[not] {} -- (1,2.2) node[xic] {};
\draw[kernels2] (0,1) {} -- (0,2.7) node[xic] {};
}

\DeclareSymbol{Xi4cap}{0}{\draw (0,1) -- (-1,2.2) node[xi] {};\draw (0,-0.25) node[not] {} -- (0,1) ; \draw[kernels2] (0,1) node[not] {} -- (1,2.2) node[xi] {};
\draw[kernels2] (0,1) {} -- (0,2.7) node[xi] {};
}

\DeclareSymbol{Xi3a}{0}{
 \draw (-1,1)  node[xi] {} -- (0,0); 
 \draw (0,0) node[xi] {}  -- (1,1) node[xi] {};
}

\DeclareSymbol{Xi4ebc1}{0}{
\draw[kernels2] (0,2) node[xi] {} -- (-1,1) ; \draw[kernels2] (-2,2)  node[xic] {} -- (-1,1) ; \draw (-1,1)  node[not] {} -- (0,0); 
 \draw (0,0) node[xic] {}  -- (1,1) node[xi] {};
}

\DeclareSymbol{Xi4ebc2}{0}{
\draw[kernels2] (0,2) node[xi] {} -- (-1,1) ; \draw[kernels2] (-2,2)  node[xi] {} -- (-1,1) ; \draw (-1,1)  node[not] {} -- (0,0); 
 \draw (0,0) node[xic] {}  -- (1,1) node[xic] {};
}

\DeclareSymbol{Xi2cbispex}{0}{\draw[kernels2] (0,1) -- (0.8,2.2) node[xi] {};\draw (0,-0.25) node[xie] {} -- (0,1); \draw[kernels2] (0,1) node[not] {} -- (-0.8,2.2) node[xi] {};}

\DeclareSymbol{Xi2cbis1p}{0}{\draw (0,1) -- (-0.8,2.2) node[xi] {};\draw (0,-0.25) node[not] {} -- (0,1) node[xi] {}; }

\DeclareSymbol{Xi2Xp}{-2}{\draw (0,-0.25) node[not] {} -- (-1,1) node[xix] {};} 

\DeclareSymbol{I1XiIXib}{0}{\draw  (0,-0.25) node[xi] {} -- (0,1) node[not] {};
\draw[kernels2] (0,1) -- (0,2.25) ; \draw (0,2.25) node[xi]{}; }

\DeclareSymbol{IXi2b}{0}{\draw  (0,-0.25) node[xi] {} -- (0,1) node[not] {};
\draw (0,1) -- (0,2.25) ; \draw (0,2.25) node[xi]{}; }

\DeclareSymbol{IXi2bex}{0}{\draw  (0,-0.25) node[xi] {} -- (0,1) node[xie] {};
\draw (0,1) -- (0,2.25) ; \draw (0,2.25) node[xi]{}; }

 \def\1{\mathbf{\symbol{1}}}

\def\one{\mathbf{1}}

\DeclareSymbol{diff}{0}{
\draw (0,0.5) node[diff] {};
}

\DeclareSymbol{diff1}{0}{
\draw (0,0.5) node[diff1] {};
}

\DeclareSymbol{diff2}{0}{
\draw (0,0.5) node[diff2] {};
}

\DeclareSymbol{geo}{0}{
\draw (0,0) node[diff] {};
\draw (0.3,0) node[diff] {};
}

\DeclareSymbol{generic}{0}{
\draw (0,0.6) node[xi] {};
}

\DeclareSymbol{g}{0}{
\draw (0,0.6) node[g] {};
}

\DeclareSymbol{Ito}{0}{
\draw (0,0.6) node[xies] {};
}

\DeclareSymbol{Itob}{0}{
\draw (0,0.6) node[xiesf] {};
}

\DeclareSymbol{greycirc}{0}{
\draw (0,0.3) node[xi] {};
}

\DeclareSymbol{not}{0}{
\draw (0,0.6) node[not] {};
\draw[tinydots] (0,0.6) circle (0.8);
}

\DeclareSymbol{genericb}{0}{
\draw (0,0.6) node[xic] {};
}

\DeclareSymbol{bluecirc}{0}{
\draw (0,0.3) node[xic] {};
}

\DeclareSymbol{genericxix}{0}{
\draw (0,0.6) node[xix] {};
}

\DeclareSymbol{genericX}{0}{
\draw (0,0.6) node[X] {};
}

\DeclareSymbol{diffIto}{1}{
\draw  (0,2.5) -- (0,0) ;
\draw (0,-0.1) node[diff] {};
\draw (0,2.5) node[xies] {};
}
\DeclareSymbol{Itodiff}{2}{
\draw(0,2.9) -- (0,-0.2);
\draw (0,2.9) node[diff] {};
\draw (0,-0.1) node[xies] {};
}

\DeclareSymbol{diffgeneric}{1}{
\draw  (0,2.5) -- (0,0) ;
\draw (0,-0.1) node[diff] {};
\draw (0,2.5) node[xi] {};
}

\DeclareSymbol{genericdiff}{2}{
\draw(0,2.9) -- (0,-0.2);
\draw (0,2.9) node[diff] {};
\draw (0,-0.1) node[xi] {};
}

\DeclareSymbol{diffdot}{2}{
\draw  (0,3) -- (0,-0.1) ;
\draw (0,3) node[not] {};
\draw (0,-0.1) node[diff] {};
}

\DeclareSymbol{diffdotmini}{0}{
\draw  (0,0) -- (0,1.2) ;
\draw (0,1.2) node[not] {};
\draw (0,0) node[diffmini] {};
}

\DeclareSymbol{dotdiff}{2}{
\draw[kernelsmod]  (0,3) -- (0,-0.1) ;
\draw (0,3) node[diff] {};
\draw (0,-0.1) node[not] {};
}

\DeclareSymbol{dotdiff1}{2}{
\draw[kernelsmod]  (0,3) -- (0,-0.1) ;
\draw (0,3) node[diff1] {};
\draw (0,-0.1) node[not] {};
}

\DeclareSymbol{dotdiff1mini}{0}{
\draw[kernelsmod]  (0,1.2) -- (0,0) ;
\draw (0,1.2) node[diffmini] {};
\draw (0,0) node[not] {};
}

\DeclareSymbol{dotdiff2}{2}{
\draw (0,3) -- (0,-0.1) ;
\draw (0,3) node[diff] {};
\draw (0,-0.1) node[not] {};
}

\DeclareSymbol{dotdiff2mini}{0}{
\draw (0,1.2) -- (0,0) ;
\draw (0,1.2) node[diffmini] {};
\draw (0,0) node[not] {};
}

\DeclareSymbol{dotdiffstraight}{0}{
\draw  (0,3) -- (0,-0.1) ;
\draw (0,3) node[diff] {};
\draw (0,-0.1) node[not] {};
}

\DeclareSymbol{arbre1}{0}{
\draw  (0,0) -- (1.5,1.5) ;
\draw (1.5,1.5) node[not] {};
\draw (0,0) node[not] {};
}

\DeclareSymbol{arbre2}{0}{
\draw  (0,0) -- (1.5,1.5) ;
\draw[kernelsmod] (0,0) -- (-1.5,1.5);
\draw (1.5,1.5) node[not] {};
\draw (0,0) node[not] {};
\draw (-1.5,1.5) node[xi] {};
}

\DeclareSymbol{arbre3}{0}{
\draw  (0,0) -- (1.5,1.5) ;
\draw[kernelsmod] (1.5,1.5) -- (0,3);
\draw (0,0) node[not] {};
\draw (1.5,1.5) node[not] {};
\draw (0,3) node[xi] {};
}

\DeclareSymbol{treeeval}{0}{
\draw (0,0) -- (1,1);
\draw (0,0) node[xi] {};
\draw (1.25,1.25) node[xi] {};
\draw (-0.6,0.6) node[]{\tiny{$i$}};
\draw (0.65,1.85) node[]{\tiny{$j$}};
}

\DeclareSymbol{testeval}{0}{
\draw (0,0) -- (1,1);
\draw (0,0) -- (-1,1);
\draw (0,0) node[xi] {};
\draw (1.25,1.25) node[xi] {};
\draw (-1.25,1.25) node[xi] {};
\draw (-0.6,-0.6) node[]{\tiny{$i$}};
\draw (0.65,1.85) node[]{\tiny{$j$}};
\draw (-1.95,1.85) node[]{\tiny{$k$}};
}

\DeclareSymbol{treeeval2}{0}{
\draw[kernelsmod] (-0.25,-1) -- (1,0.5) ;
\draw[kernelsmod] (1,0.5) -- (-0.25,2);
\draw (1,0.5) node[diff2] {};
\draw (-0.25,-1) node[not] {};
\draw (-0.25,2) node[xi] {};
\draw (-0.6,1.2) node[]{\tiny{1}};
}

\DeclareSymbol{arbreact}{1}{
\draw (0,0) node[not] {};
\draw[kernelsmod] (0,0) -- (1,1);
\draw[kernelsmod] (0,0) -- (-1,1);
\draw (-1,1) node[xic] {};
\draw  (0,2) -- (1,1) ;
\draw (0,2) node[xic] {};
\draw (1,1) node[xi] {};
}

\DeclareSymbol{arbreact1}{0}{
\draw (0,-1.5) -- (0,0);
\draw[kernelsmod] (0,0) -- (1,1);
\draw[kernelsmod] (0,0) -- (-1,1);
\draw  (0,2) -- (1,1) ;
\draw (0,-1.5) node[diff] {};
\draw (0,0) node[not] {};
\draw (-1,1) node[xic] {};
\draw (0,2) node[xic] {};
\draw (1,1) node[xi] {};
}

\DeclareSymbol{arbreact2}{0}{
\draw (0,-0.75) -- (-1,0.5); 
\draw (0,-0.75) -- (1,0.5);
\draw (0,1.5) -- (1,0.5);
\draw (0,1.5) node[xic] {};
\draw (1,0.5) node[xi] {};
\draw (-1,0.5) node[xic] {};
\draw (0,-0.75) node[diff] {};
}

\DeclareSymbol{arbreact3}{0}{
\draw[kernelsmod] (0,-0.75) -- (-1,0.5); 
\draw[kernelsmod] (0,-0.75) -- (1,0.5);
\draw (0,1.75) -- (1,0.5);
\draw (2,1.75) -- (1,0.5);
\draw (0,1.75) node[xic] {};
\draw (1,0.5) node[diff] {};
\draw (-1,0.5) node[xic] {};
\draw (2,1.75) node[xi] {};
\draw (0,-0.75) node[not] {};
}

\DeclareSymbol{pre_im_I1Xitwo}{0}{
\draw[kernels2] (0,-0.3) node[not] {} -- (-0.6,0.7) ;
\draw[kernels2] (0,-0.3) -- (0.6,0.7);
\draw (0,0.9) node[g] {};
}

\DeclareSymbol{pre_im_cI1Xi4ab}{2}{
\draw[kernels2] (0,-1) node[not] {} -- (-0.6,0) ;
\draw[kernels2] (0,-1) -- (0.6,0);
\draw (0,0.2) node[g] {};
\draw (0,0.6) -- (0,1.5);
\draw[kernels2] (0,1.5) node[not] {} -- (-0.6,2.5) ;
\draw[kernels2] (0,1.5) -- (0.6,2.5);
\draw (0,2.7) node[g] {};
}

\DeclareSymbol{pre_im_I1Xi4acc2}{0}{
\draw[kernels2] (-1,-0.5) node[not] {} -- (-1.6,0.5) ;
\draw[kernels2] (-1,-0.5) -- (-0.4,0.5);
\draw[kernels2] (-1,-0.5) -- (0.2,-1.5) node[not] {} ;
\draw (-1,1.1) -- (-1,2);
\draw[kernels2] (0.2,-1.5) -- (0.2,2);
\draw (-1,0.7) node[g] {};
\draw (-0.3,2.2) node[g] {};
}

\DeclareSymbol{pre_im_I1Xi4abcc2}{2}{
\draw[kernels2] (0,-1) node[not] {} -- (-1,0) node[not] {};
\draw[kernels2] (-1,1.2) node[not] {} -- (-1,0);
\draw[kernels2] (-1,1.2) -- (-1.5,2.5);
\draw[kernels2] (-1,1.2) -- (-0.5,2.5);
\draw[kernels2] (-1,0) -- (0.7,2.5);
\draw[kernels2] (0,-1) -- (1.5,2.5);
\draw (-1,2.7) node[g] {};
\draw (1,2.7) node[g] {};
}

\DeclareSymbol{pre_im_2I1Xi4c1}{2}{
\draw[kernels2] (0,-0.5) node[not] {} -- (-1,0.5) node[not] {};
\draw[kernels2] (0,-0.5) -- (1,0.5) node[not] {};
\draw[kernels2] (-1,0.5) node[not] {}-- (-1.7,2);
\draw[kernels2]  (-1,2) -- (1,0.5);
\draw[kernels2] (-1,0.5) -- (1,2);
\draw[kernels2] (1,0.5) -- (1.7,2);
\draw (-1.2,2.2) node[g] {};
\draw (1.2,2.2) node[g] {};
}

\DeclareSymbol{pre_im_Xi4eabisc2}{2}{
\draw[kernels2] (1.2,-0.5) node[not] {} -- (-0.7,0.8) ;
\draw[kernels2] (1.2,-0.5) -- (0.4,0.8);
\draw (0,1.4)  -- (0,2.2);
\draw (1.2,2.2) -- (1.2,-0.6);
\draw (0,1) node[g] {};
\draw (0.6,2.4) node[g] {};
}

\DeclareSymbol{pre_im_Xi4eabisc22}{2}{
\draw (1.2,-0.5) node[not] {} -- (-0.7,0.8) ;
\draw[kernels2] (1.2,-0.5) -- (0.4,0.8);
\draw (0,1.4)  -- (0,2.2);
\draw[kernels2] (1.2,2.2) -- (1.2,-0.6);
\draw (0,1) node[g] {};
\draw (0.6,2.4) node[g] {};
}

\DeclareSymbol{pre_im_Xi4eabisc222}{2}{
\draw[kernels2] (0.4,-0.5) node[not] {} -- (-0.6,1) ;
\draw[kernels2] (1.2,0) -- (0.3,1);
\draw (0,1.1)  -- (0,2.5);
\draw[kernels2] (1.2,2.5) -- (1.2,0) node[not] {} -- (0.4,-0.6);
\draw (0,1.2) node[g] {};
\draw (0.6,2.5) node[g] {};
}

\DeclareSymbol{pre_im_Xi4eabc2}{2}{
\draw (0,-0.5) node[not] {} -- (-1,0.5) node[not] {};
\draw[kernels2] (-1,0.5) -- (-1.5,2);
\draw[kernels2] (0,-0.5)  -- (0.7,2);
\draw[kernels2] (-1,0.5) -- (-0.5,2);
\draw[kernels2] (0,-0.5) -- (1.5,2);
\draw (-1,2.2) node[g] {};
\draw (1,2.2) node[g] {};
}

\DeclareSymbol{pre_im_Xi4eabbisc2}{2}{
\draw[kernels2] (0,-0.5) node[not] {} -- (-1,0.5) node[not] {};
\draw[kernels2] (-1,0.5) -- (-1.5,2);
\draw[kernels2] (0,-0.5)  -- (0.7,2);
\draw[kernels2] (-1,0.5) -- (-0.5,2);
\draw (0,-0.5) -- (1.5,2);
\draw (-1,2.2) node[g] {};
\draw (1,2.2) node[g] {};
}

\DeclareSymbol{pre_im_I1Xi4abcc1}{2}{
\draw[kernels2] (0,-1) node[not] {} -- (-1,0) node[not] {};
\draw[kernels2] (0,1.1) node[not] {} -- (-1,0);
\draw[kernels2] (-1,0) -- (-1.5,2.5);
\draw[kernels2] (0,1.1) node[not] {} -- (-0.5,2.5);
\draw[kernels2] (0,1.1) -- (0.5,2.5);
\draw[kernels2] (0,-1) -- (1.5,2.5);
\draw (-1,2.7) node[g] {};
\draw (1,2.7) node[g] {};
}

\DeclareSymbol{pre_im_Xi4eabc1}{2}{
\draw (0,-0.5) node[not] {} -- (-1,0.5) node[not] {};
\draw[kernels2] (-1,0.5) -- (-1.5,2);
\draw[kernels2] (0,-0.5)  -- (-0.5,2);
\draw[kernels2] (-1,0.5) -- (0.8,2);
\draw[kernels2] (0,-0.5) -- (1.5,2);
\draw (-1,2.2) node[g] {};
\draw (1,2.2) node[g] {};
}

\DeclareSymbol{pre_im_Xi4ba1b}{2}{
\draw[kernels2] (0,0) node[not] {}  -- (1.8,1.5);
\draw[kernels2] (0,0) -- (0.8,1.5);
\draw (0,-0.1) -- (-1.8,1.5);
\draw (0,-0.1) -- (-0.8,1.5);
\draw (-1,1.7) node[g] {};
\draw (1,1.7) node[g] {};
}

\DeclareSymbol{pre_im_Xi4ba2}{2}{
\draw (0,-0.1) node[not] {}  -- (1.8,1.5);
\draw[kernels2] (0,0) -- (0.8,1.5);
\draw (0,-0.1) -- (-1.8,1.5);
\draw[kernels2] (0,0) -- (-0.8,1.5);
\draw (-1,1.7) node[g] {};
\draw (1,1.7) node[g] {};
}

\DeclareSymbol{pre_im_Xi4cabc2}{2}{
\draw[kernels2] (0,-0.5) node[not] {} -- (-1,0.5) node[not] {};
\draw[kernels2] (-1,0.5) -- (-1.5,2);
\draw (-1,0.5)  -- (0.7,2);
\draw[kernels2] (-1,0.5) -- (-0.5,2);
\draw[kernels2] (0,-0.5) -- (1.5,2);
\draw (-1,2.2) node[g] {};
\draw (1,2.2) node[g] {};
}

\DeclareSymbol{pre_im_Xi4cabc1}{2}{
\draw[kernels2] (0,-0.5) node[not] {} -- (-1,0.5) node[not] {};
\draw (-1,0.5) -- (-1.5,2);
\draw[kernels2] (-1,0.5)  -- (0.7,2);
\draw[kernels2] (-1,0.5) -- (-0.5,2);
\draw[kernels2] (0,-0.5) -- (1.5,2);
\draw (-1,2.2) node[g] {};
\draw (1,2.2) node[g] {};
}

\DeclareSymbol{pre_im_Xi4eabbisc1}{2}{
\draw[kernels2] (0,-0.5) node[not] {} -- (-1,0.5) node[not] {};
\draw[kernels2] (-1,0.5) -- (-1.5,2);
\draw[kernels2] (0,-0.5)  -- (-0.5,2);
\draw[kernels2] (-1,0.5) -- (0.8,2);
\draw (0,-0.5) -- (1.5,2);
\draw (-1,2.2) node[g] {};
\draw (1,2.2) node[g] {};
}

\DeclareSymbol{pre_im_1}{0}{
\draw[kernels2] (0,-0.5) node[not] {} -- (-0.6,0.5) ;
\draw[kernels2] (0,-0.5) -- (0.6,0.5);
\draw (0,1.1)  -- (-0.55,2);
\draw (0,1.1)  -- (0.55,2);
\draw (0,0.7) node[g] {};
\draw (0,2.2) node[g] {};
}

\DeclareSymbol{disconnect}{0}{
\draw[kernels2] (0,-0.5) node[not] {} -- (-0.6,0.5) ;
\draw[kernels2] (0,-0.5) -- (0.6,0.5);
\draw (-0.55,1.1)  -- (-0.55,2.3);
\draw (0.55,2.3) -- (0.55,1.5) -- (1.2,1.5) -- (1.2,3.5) -- (0.55,3.5) -- (0.55,2.7);
\draw (0,0.7) node[g] {};
\draw (0,2.5) node[g] {};
}

\DeclareSymbol{pre_im_2}{2}{\draw[kernels2] (0,0) node[not] {} -- (-1,1) node[not] {};
\draw[kernels2] (0,0) -- (1,1) node[not] {};
\draw[kernels2] (-1,1) -- (-1.5,2.5);
\draw[kernels2] (-1,1) -- (-0.5,2.5);
\draw[kernels2] (1,1) -- (0.5,2.5);
\draw[kernels2] (1,1) -- (1.5,2.5);
\draw (-1,2.7) node[g] {};
\draw (1,2.7) node[g] {};
}

\DeclareSymbol{CX_rec}{0}{
\draw [black] (-0.3,1) to (-0.3,-0.3);
\draw [black] (0.3,1) to (0.3,-0.3);
\draw [black] (-0.3,1) to (-0.3,2.3);
\draw [black] (0.3,1) to (0.3,2.3);
\draw (0,1) node[rec] {};
}

\DeclareSymbol{CX_cerc}{0}{
\draw [black] (0,1) to (0,-0.3);
\draw (0,1) node[cerc] {};
}

\DeclareMathAlphabet{\mathpzc}{OT1}{pzc}{m}{it}

\def\eqref#1{(\ref{#1})}

\makeatletter 
\newcommand*{\bigcdot}{}
\DeclareRobustCommand*{\bigcdot}{%
  \mathbin{\mathpalette\bigcdot@{}}%
}
\newcommand*{\bigcdot@scalefactor}{.5}
\newcommand*{\bigcdot@widthfactor}{1.15}
\newcommand*{\bigcdot@}[2]{%
  \sbox0{$#1\vcenter{}$}
  \sbox2{$#1\cdot\m@th$}%
  \hbox to \bigcdot@widthfactor\wd2{%
    \hfil
    \raise\ht0\hbox{%
      \scalebox{\bigcdot@scalefactor}{%
        \lower\ht0\hbox{$#1\bullet\m@th$}%
      }%
    }%
    \hfil
  }%
}
\makeatother

\tcbset
{colframe=boxcolor,colback=symbols!7!pagebackground,coltext=pageforeground,
fonttitle=\bfseries,nobeforeafter,center title,size=fbox,boxsep=1.5pt,
top=0mm,bottom=0mm,boxsep=0mm,tcbox raise base}

\def\two{{\<generic>\kern0.05em\<genericb>}}
\def\twoI{{\<Ito>\kern0.05em\<Itob>}}

\def\st{\mathsf{fgt}}
\def\mail#1{\burlalt{#1}{mailto:#1}}

\usepackage{thmtools} 

\begin{document}

\def\st{\mathsf{fgt}}
\def\mail#1{\burlalt{#1}{mailto:#1}}
\title{Post-Lie algebras in Regularity Structures}
\author{Yvain Bruned$^1$, Foivos Katsetsiadis$^2$}
\institute{ 
 IECL (UMR 7502), Université de Lorraine
 \and University of Edinburgh \\
Email:\ \begin{minipage}[t]{\linewidth}
\mail{yvain.bruned@univ-lorraine.fr},
\\ \mail{F.I.Katsetsiadis@sms.ed.ac.uk}.
\end{minipage}}
\def\dsqcup{\sqcup\mathchoice{\mkern-7mu}{\mkern-7mu}{\mkern-3.2mu}{\mkern-3.8mu}\sqcup}

\maketitle

\begin{abstract}

\ \ \ \ In this work, we construct the deformed Butcher-Connes-Kreimer  Hopf algebra coming from the theory of Regularity Structures as the universal envelope of a post-Lie algebra. We show that this can be done using either of the two combinatorial structures that have been proposed in the context of singular SPDEs: decorated trees and multi-indices. Our construction is inspired from multi-indices where the Hopf algebra was obtained as the universal envelope of a Lie algebra and it has been proved that one can find a basis that is symmetric with respect to certain elements. We show that this Lie algebra comes from an underlying post-Lie structure.

\end{abstract}

\setcounter{tocdepth}{1}
\tableofcontents

\section{ Introduction }

 \ \ \ \ \ \  Regularity Structures were introduced by Martin Hairer in \cite{reg} and are nowadays able to provide well-posedness to a large class of singular stochastic partial differential equations (SPDEs). This is performed via the theory developed in \cite{BHZ,CH16,BCCH}, which may be seen as a black box that constructs Taylor-type expansions of solutions to these singular dynamics. These expansions rely on a description by means of decorated trees that provide an abstract representation of the iterated integrals appearing in the series expansion of the solution in the smooth setting. In \cite{BHZ}, analytical operations on these expansions such as recentering and renormalisation are performed via Hopf algebras which are close in spirit to the Butcher-Connes-Kreimer Hopf algebra \cite{Butcher72,CK1,CK2} (recentering) and the extraction-contraction Hopf algebra \cite{CEM} (renormalisation). One of the crucial points of this algebraic approach is the cointeraction between the two Hopf algebras involved in \cite{reg,BHZ} which is reminiscent of the cointeraction proven in \cite{CEM} and observed at a group theoretic level in numerical analysis. The reader is referred to \cite{CHV10} for a review of its applications to numerical analysis. We also refer the reader to \cite{FrizHai,BaiHos} for long surveys and to \cite{EMS} for a short survey on the theory of Regularity Structures.

 In this paper, we shall provide a strong link between the Hopf algebras involved in the theory and the notion of post-Lie algebra. Post-Lie algebras appear naturally in the context of an affine connection with constant torsion and vanishing curvature (see \cite{ML13,ELM}). They were first mentioned in \cite{Val,ML08} on the partition of posets and in the context of Lie-Butcher series. They have also been used in many works in numerical analysis (see \cite{ML13,FM,CEO,AEM,AEMM}).
 
  We begin by explaining how this type of structure can appear in the context of singular SPDEs.
In \cite{BM22}, a deformation of the grafting pre-Lie product was introduced by Bruned and Manchon that gives the pre-Lie product defined in \cite{BCCH}. Then, using the result of Guin-Oudom  \cite{Guin1,Guin2}, the authors construct the Hopf algebra that encodes the combinatorics of recentering by taking a suitable quotient of the Lie envelope of the plugging pre-Lie algebra. This procedure produces a deformation of the Grossman-Larson product on trees \cite{GL}. Then, the extraction-contraction Hopf algebra given in \cite{BHZ} is obtained from this deformed product. Indeed, it produces a pre-Lie product and together with the Guin-Oudom procedure, one obtains the extraction-contraction Hopf algebra. This approach also gives the cointeraction at the level of the deformed pre-Lie products. 
This new deformation formalism has been crucial in \cite{BB21,BB21b} for providing a simple proof of the renormalised equation that works in the non-translation invariant setting. The proof relies on the local renormalisation maps introduced \cite{BR18}. It also inspires the development of the Hopf algebra of multi-indices suitable for quasilinear SPDEs in \cite{OSSW,LOT,LOTT}. In particular,  in \cite{LOT}, the authors construct a product operation on the universal enveloping algebra of a Lie algebra of derivations by using a procedure that is similar in spirit to the one developed by Guin and Oudom.

 \ \ The main contribution of this paper is to unify the construction of the Hopf algebras appearing in the theory of Regularity Structures as well as their presentation by means of a distinguished basis, by viewing them as the Lie enveloping algebras of suitable post-Lie algebras and to show that the construction of the associative product on the universal envelope of $L$ in \cite{LOT} together with the subsequent attainment of a partially symmetric basis, can also be seen as taking the universal envelope of a particular post-Lie algebra and exploiting the isomorphism induced by virtue of the result in \cite{ELM}. The post-Lie perspective provides a cleaner construction than the one given in \cite{BM22} where one had to make the appropriate identifications after taking the universal envelope over the plugging pre-Lie algebra. Our construction replaces the plugging operation by a post-Lie structure that is intimately related to the deformed grafting product. 
  
  The appearance of a post-Lie structure reflects the fact that the formal differentiation operators obey certain non-commutative relations which necessitate the move from pre-Lie algebras to this more general setting. It is our contention that this idea opens new perspectives in the field due to the intricate interplay between the combinatorial, algebraic and geometric nature of the post-Lie category. A post-Lie algebra carries the geometric datum of a connection and at the same time produces an associated representation for its Lie enveloping Hopf algebra as well as an Hopf isomorphism while, in practice, the results attained herein have a distinctly combinatorial flavour. Indeed, the special case of a pre-Lie algebra has already been used by Loday and Ronco in \cite{LR} to formalise the idea of a combinatorial Hopf algebra by means of a categorical equivalence. Lastly, our use of the post-Lie structure has already had some impact by allowing one to construct planar Regularity Structures in \cite{Rahm} for SPDEs in homogeneous spaces as well as to revisit, in \cite{BF23}, the Chapoton-Foissy isomorphism given in \cite{Foi02,cha10}. Also, in \cite{JZ}, the authors were inspired by our construction to study a class of post-Lie algebras comprised of derivations. 
 
We can state our main result as a meta theorem:
\begin{theorem} \label{main_theorem}
The Hopf algebra in \cite{reg,BHZ} encoding the combinatorics of recentering is obtained as the universal envelope of a suitable post-Lie algebra.
\end{theorem}
 Theorem~\ref{main_theorem} is split in the sequel into several results: Theorem~\ref{main_result_trees} and  Theorem~\ref{identification_star_star_2} for decorated trees and Theorem~\ref{main_result_multi_indices} for multi-indices.
The central idea behind Theorem~\ref{main_theorem} is the non-commutation of derivatives observed for multi-indices \cite{LOT} and in \cite{BaiHos} for the proof of the renormalised equation. Indeed, one has a collection of derivatives $ D^{(n)} $ and $ \partial_i $ with $ n \in \mathbb{N}^{d+1} $ and $ i, j \in \lbrace 0,...,d \rbrace $ such that:
\begin{equs} \label{non_commutation_intro}
D^{(n)}  D^{(m)} & = D^{(m)} D^{(n)}, \quad \partial_i \partial_j =   \partial_j \partial_i \\
\partial_i D^{(n)} & = D^{(n)} \partial_i + n_i D^{(n-e_i)}
\end{equs}
where the $ e_i $ are the canonical basis of $ \mathbb{N}^{d+1} $. The last identity gave the authors the inspiration to introduce a suitable Lie bracket on derivations that reflects this property. Then, one has to look for a product compatible with this Lie bracket that will be a post-Lie product. In the context of decorated trees, one has an analogue of \eqref{non_commutation_intro} given by:
\begin{equs}
 \label{non_commutation_intro_2}
\uparrow^{i}_{N_\tau} \left( \sigma \widehat{\curvearrowright}^n \tau  \right) =    \sigma \widehat{\curvearrowright}^n \, ( \uparrow^{i}_{N_{\tau}} \tau)  - \sigma \widehat{\curvearrowright}^{n-e_i} \tau.
\end{equs}
Here $ \sigma, \tau $ are decorated trees with decorations on the nodes and edges given by $ \mathbb{N}^{d+1} $. The operator $ \uparrow^{i}_{N_\tau} $ sums over all the possible ways to increase a node decoration in $ \tau $ by $ e_i $. The product $ \widehat{\curvearrowright}^{n-e_i} $ is a deformed grafting operation. 
 The reader should keep in mind the following dictionary:
 
\begin{equs}
D^{(n)} \equiv \widehat{\curvearrowright}^n, \quad \uparrow^{i}_{N_\tau}  \equiv \partial_i.
\end{equs}

Hence, the statement of Theorem~\eqref{main_theorem} can be seen as independent of the underlying choice of formalism whether that be decorated trees or multi-indices. Theorem~\ref{main_theorem} together with the deformation formalism developed in \cite{BM22} yield a precise answer on how to build up the structures first proposed in \cite{BHZ} both on multi-indices and decorated trees. Among other things, we also expect this formalism to prove very fruitful in exporting algebraic properties from numerical analysis and perturbative quantum field theory to singular SPDEs.

 Finally, let us outline the paper by summarising the content of its sections. In Section~\ref{section::Post_Lie}, we recall the basics of post-Lie algebras with Definition~\ref{definition_post_lie} and the Guin-Oudom type procedure on such a product that leads to Theorem~\ref{main_theorem_section_2}. It establishes a Hopf algebra isomorphism between the Hopf algebra equipped with the product obtained from the post-Lie structure and the universal enveloping algebra of a well-chosen Lie algebra. In Section~\ref{section_trees_multi_indices}, we introduce decorated trees and multi-indices. On decorated trees, we recall multi-grafting products and their deformation coming from \cite{BM22}. We stress the crucial non-commutative property in Proposition~\ref{prop_non_com} between the deformed grafting product and the insertion-of-decorations operator. This also has an analogue at the level of multi-indices with certain derivative operators. In Section 4, we make precise Theorem~\ref{main_theorem} in the context of decorated trees by introducing the appropriate Lie algebra and post-Lie product. This allows us to apply the result from \cite{ELM} and obtain the desired isomorphism. As a consequence, we obtain a partially symmetric basis for the corresponding Lie enveloping algebra. In Section 5, we identify the appropriate post-Lie structure in the context of multi-indices and apply the same procedure, again obtaining the isomorphism result afforded to us by \cite{ELM}. This gives an alternative way to obtain a partially symmetric basis for $U(L)$, as was obtained in \cite{LOT}.

\subsection*{Acknowledgements}

{\small
	The authors are very grateful to the referees for their careful reading of the manuscript which led  to substantial improvements in the clarity of the exposition.
The authors thank Pablo Linares, Felix Otto, Markus Tempelmayr and  Pavlos Tsatsoulis for interesting discussions on the topic of multi-indices.
Y. B. thanks the Max Planck Institute for Mathematics in the Sciences (MiS) in Leipzig for having supported his research via a long stay in Leipzig from January to June 2022. Y. B. is funded by the ANR via the project LoRDeT (Dynamiques de faible régularité via les arbres décorés) from the projects call T-ERC\_STG. F. K. thanks the Max Planck Institute for Mathematics in the Sciences (MiS) for a short stay in Leipzig during which this work started.
}

\section{ Post-Lie algebras }
\label{section::Post_Lie}
In this section, we briefly recall the definition of a post-Lie algebra and its various properties.
\begin{definition} \label{definition_post_lie}
A post-Lie algebra is a Lie algebra $ (\mathfrak{g}, [.,.]) $ equipped with a bilinear product $ \triangleright $ satisfying the following identities:
\begin{equs} \label{ident1}
x \triangleright [y,z] = [x \triangleright y,z] + [y, x \triangleright z]
\end{equs}
and
\begin{equs} \label{ident2}
[x,y] \triangleright z = a_{\triangleright}(x,y,z) - a_{\triangleright}(y,x,z)
\end{equs}
with $ x,y,x \in \mathfrak{g} $ and the associator $ a_{\triangleright}(x,y,z) $ is given by:
\begin{equs}
a_{\triangleright}(x,y,z) = x \triangleright (  y  \triangleright z ) - (x \triangleright y) \triangleright z.
\end{equs}
\end{definition}
 When $ \mathfrak{g} $ is an abelian Lie algebra,
\eqref{ident1} is void and 	$ \eqref{ident2} $ is the classical definition of the pre-Lie relation. The equation  gives then a (in general) non-trivial Lie bracket. One can define a new Lie bracket $[[ .,.]]$ given by:
\begin{equation*}
[[x,y]] = [x,y] + x \triangleright y - y \triangleright x.
\end{equation*}
The post-Lie product $ \triangleright $  can be extended to a product on the universal enveloping algebra $ U(\mathfrak{g}) $ by first defining it on $ \mathfrak{g}  \otimes U(\mathfrak{g})$:
\begin{equs}
x \triangleright \one = 0, \quad x \triangleright y_1 ... y_n = \sum_{i=1}^n y_1 ... (x  \triangleright y_i) ... y_n.
\end{equs}
and then extending it to $ U(\mathfrak{g}) \otimes U(\mathfrak{g}) $ by defining:
\begin{equs}
\one  \triangleright A & = A, \quad
x A  \triangleright y = x \triangleright (A \triangleright y) -(x \triangleright A) \triangleright y, \\
A \triangleright B C & = \sum_{(A)} (A^{(1)} \triangleright B)(A^{(2)} \triangleright C).
\end{equs}
where $ A, B, C \in U(\mathfrak{g}) $ and $ x, y \in \mathfrak{g} $.  Here, we have used Sweedler's notation for the coproduct $ \Delta $:  $
\Delta A =\sum_{(A)} A^{(1)}\otimes A^{(2)} $.  This coproduct is defined for $ x \in \mathfrak{g} $ by:
\begin{equs}
\Delta x = x \otimes \one + \one \otimes x
\end{equs}
and then extended multiplicatively with respect to the concatenation product. Finally, one is able to define an associative product $*$ on $U(\mathfrak{g})$:
\begin{equs} \label{product_1}
A * B = \sum_{(A)} A^{(1)} (A^{(2)} \triangleright B).
\end{equs}

Then, one of the main results in \cite[Thm 3.4]{ELM}, which depends on the above construction, allows us to exploit the underlying post-Lie structure on $\mathfrak{g}$ in order to obtain information about the structure of $U(\mathfrak{g})$.

\begin{theorem} \label{main_theorem_section_2}
The Hopf algebra $ (U(\mathfrak{g}),*,\Delta) $ is isomorphic to the enveloping
algebra $ U(\bar{\mathfrak{g}}) $ where $ \bar{\mathfrak{g}} $ is the Lie algebra equipped with the Lie bracket $ [[.,.]] $.
\end{theorem} 

\begin{remark}
This result is a generalisation of the Guin-Oudom  procedure in \cite{Guin1,Guin2} applied on a pre-Lie product. The Guin-Oudom construction allows one to get an associative product and an isomorphism with the enveloping algebra of a Lie algebra whose bracket is obtained by antisymmetrisation of a pre-Lie product. As in the simpler pre-Lie case, the post-Lie assumption gives some extra structure to the underlying Lie algebra. The construction of the $*$ product can be viewed as a way to upload the extra structure to the universal enveloping algebra and exploit the additional information by means of an isomorphism theorem.
\end{remark}
Furthermore, one has the following proposition which shows that a post-Lie structure underlying a Lie algebra produces a representation-theoretic datum (see \cite[Sec. 5]{MQS}):
\begin{proposition}\label{rep_datum}
The mapping 
$$
\rho : \mathfrak{\bar{g}} \rightarrow End_{\mathbb{R}}(U(\mathfrak{g})), \quad 
x \mapsto ( y \mapsto x \triangleright y + x y)
$$
is a linear representation of the Lie algebra $\mathfrak{\bar{g}}$. Furthermore, by the universal property of the enveloping algebra it extends to a linear representation 
\begin{equs}
\rho:U(\mathfrak{\bar{g}}) \rightarrow End_{\mathbb{R}} \(U(\mathfrak{g})\)
\end{equs}
of the algebra $U(\mathfrak{\bar{g}})$. We call this the representation induced by the post-Lie algebra $(\mathfrak{g}, [ .,.], \triangleright)$ or, more shortly, the induced representation.
\end{proposition}

We finish this section by presenting standard examples of pre- and post-Lie algebras. We first consider the set of non-planar rooted trees denoted by $ \CT $.  We define a product $ \bar{\curvearrowright}  $ called grafting product defined for two rooted trees $ \tau, \sigma $ by:
\begin{equation*} 
\sigma \bar{\curvearrowright} \tau:=\sum_{v\in  N_{\tau} } \sigma \bar{\curvearrowright}_v  \tau,
\end{equation*}
where $ N_\tau $ is the set of vertices of $ \tau $ and $\sigma \bar{\curvearrowright}_v \tau$ is obtained by grafting the tree $\sigma$ on the tree $\tau$ at vertex $v$ by means of a new edge. Below, we provide an example of computation for $ \bar{\curvearrowright} $
\begin{equs} 
	\bullet \, \bar{\curvearrowright}  \begin{tikzpicture}[scale=0.2,baseline=0.1cm]
		\node at (0,0)  [dot,label= {[label distance=-0.2em]below: \scriptsize  $     $} ] (root) {};
		\node at (1,2)  [dot,label={[label distance=-0.2em]above: \scriptsize  $  $}] (right) {};
		\node at (-1,2)  [dot,label={[label distance=-0.2em]above: \scriptsize  $ $} ] (left) {};
		\draw[kernel1] (right) to
		node [sloped,below] {\small }     (root); \draw[kernel1] (left) to
		node [sloped,below] {\small }     (root);
{i}	\end{tikzpicture} = \begin{tikzpicture}[scale=0.2,baseline=0.1cm]
		\node at (0,0)  [dot,label= {[label distance=-0.2em]below: \scriptsize  $     $} ] (root) {};
		\node at (0,3)  [dot,label= {[label distance=-0.2em]above: \scriptsize  $     $} ] (center) {};
		\node at (2,2)  [dot,label={[label distance=-0.2em]above: \scriptsize  $ $}] (right) {};
		\node at (-2,2)  [dot,label={[label distance=-0.2em]above: \scriptsize  $ $} ] (left) {};
		\draw[kernel1] (right) to
		node [sloped,below] {\small }     (root); 
		\draw[kernel1] (center) to
		node [sloped,below] {\small }     (root); 
		\draw[kernel1] (left) to
		node [sloped,below] {\small }     (root);
	\end{tikzpicture} + 2 \begin{tikzpicture}[scale=0.2,baseline=0.1cm]
		\node at (0,0)  [dot,label= {[label distance=-0.2em]below: \scriptsize  $   $} ] (root) {};
		\node at (0,4)  [dot,label= {[label distance=-0.2em]above: \scriptsize  $     $} ] (center) {};
		\node at (1,2)  [dot,label={[label distance=-0.2em]right: \scriptsize  $  $}] (right) {};
		\node at (-1,2)  [dot,label={[label distance=-0.2em]above: \scriptsize  $ $} ] (left) {};
		\draw[kernel1] (right) to
		node [sloped,below] {\small }     (root);
		\draw[kernel1] (center) to
		node [sloped,below] {\small }     (right); \draw[kernel1] (left) to
		node [sloped,below] {\small }     (root);
	\end{tikzpicture}. 
\end{equs}
The space $ \CT $ equipped with $ \bar{\curvearrowright}  $ is the free pre-Lie algebra over one generator $ \bullet $ (see \cite{ChaLiv}). We consider the linear span of planar rooted trees denoted by $ \mathcal{P}\mathcal{T} $. Now, one has 
\begin{equs}
	\begin{tikzpicture}[scale=0.2,baseline=0.1cm]
		\node at (0,0)  [dot,label= {[label distance=-0.2em]below: \scriptsize  $   $} ] (root) {};
		\node at (-1,4)  [dot,label= {[label distance=-0.2em]above: \scriptsize  $     $} ] (center) {};
		\node at (1,2)  [dot,label={[label distance=-0.2em]right: \scriptsize  $  $}] (right) {};
		\node at (-1,2)  [dot,label={[label distance=-0.2em]above: \scriptsize  $ $} ] (left) {};
		\draw[kernel1] (right) to
		node [sloped,below] {\small }     (root);
		\draw[kernel1] (center) to
		node [sloped,below] {\small }     (left); \draw[kernel1] (left) to
		node [sloped,below] {\small }     (root);
	\end{tikzpicture}  \neq \begin{tikzpicture}[scale=0.2,baseline=0.1cm]
	\node at (0,0)  [dot,label= {[label distance=-0.2em]below: \scriptsize  $   $} ] (root) {};
	\node at (1,4)  [dot,label= {[label distance=-0.2em]above: \scriptsize  $     $} ] (center) {};
	\node at (1,2)  [dot,label={[label distance=-0.2em]right: \scriptsize  $  $}] (right) {};
	\node at (-1,2)  [dot,label={[label distance=-0.2em]above: \scriptsize  $ $} ] (left) {};
	\draw[kernel1] (right) to
	node [sloped,below] {\small }     (root);
	\draw[kernel1] (center) to
	node [sloped,below] {\small }     (right); \draw[kernel1] (left) to
	node [sloped,below] {\small }     (root);
\end{tikzpicture} 
\end{equs}
which is not the case in the non-planar setting.
 One can define a magmatic operation $ \bar{\curvearrowright}_{l} $  called left grafting. It is given for $ \tau, \sigma $ non-planar rooted trees by:
\begin{equation*} 
	\sigma \bar{\curvearrowright}_{l} \tau:=\sum_{v\in  N_{\tau} } \sigma \bar{\curvearrowright}_{l,v}  \tau,
\end{equation*}
 where $\sigma \bar{\curvearrowright}_{l,v} \tau$ is obtained by grafting the tree $\sigma$ on the tree $\tau$ at vertex $v$ by means of a new edge at the left most location. Below, we provide an example of computation for $ \bar{\curvearrowright}_{l} $
 \begin{equs} 
 	\bullet \, \bar{\curvearrowright}_l  \begin{tikzpicture}[scale=0.2,baseline=0.1cm]
 		\node at (0,0)  [dot,label= {[label distance=-0.2em]below: \scriptsize  $   $} ] (root) {};
 		\node at (-1,4)  [dot,label= {[label distance=-0.2em]above: \scriptsize  $     $} ] (center) {};
 		\node at (1,2)  [dot,label={[label distance=-0.2em]right: \scriptsize  $  $}] (right) {};
 		\node at (-1,2)  [dot,label={[label distance=-0.2em]above: \scriptsize  $ $} ] (left) {};
 		\draw[kernel1] (right) to
 		node [sloped,below] {\small }     (root);
 		\draw[kernel1] (center) to
 		node [sloped,below] {\small }     (left); \draw[kernel1] (left) to
 		node [sloped,below] {\small }     (root);
 	\end{tikzpicture}  = \begin{tikzpicture}[scale=0.2,baseline=0.1cm]
 		\node at (0,0)  [dot,label= {[label distance=-0.2em]below: \scriptsize  $     $} ] (root) {};
 		\node at (0,2)  [dot,label= {[label distance=-0.2em]above: \scriptsize  $     $} ] (center) {};
 		\node at (0,4)  [dot,label= {[label distance=-0.2em]above: \scriptsize  $     $} ] (centerc) {};
 		\node at (1.5,1.5)  [dot,label={[label distance=-0.2em]above: \scriptsize  $ $}] (right) {};
 		\node at (-1.5,1.5)  [dot,label={[label distance=-0.2em]above: \scriptsize  $ $} ] (left) {};
 		\draw[kernel1] (right) to
 		node [sloped,below] {\small }     (root); 
 		\draw[kernel1] (center) to
 		node [sloped,below] {\small }     (root); 
 		\draw[kernel1] (center) to
 		node [sloped,below] {\small }     (centerc); 
 		\draw[kernel1] (left) to
 		node [sloped,below] {\small }     (root);
 	\end{tikzpicture} +  \begin{tikzpicture}[scale=0.2,baseline=0.1cm]
 	\node at (0,0)  [dot,label= {[label distance=-0.2em]below: \scriptsize  $   $} ] (root) {};
 	\node at (0,4)  [dot,label= {[label distance=-0.2em]above: \scriptsize  $     $} ] (center) {};
 	\node at (-2,4)  [dot,label= {[label distance=-0.2em]above: \scriptsize  $     $} ] (centerl) {};
 	\node at (1,2)  [dot,label={[label distance=-0.2em]right: \scriptsize  $  $}] (right) {};
 	\node at (-1,2)  [dot,label={[label distance=-0.2em]above: \scriptsize  $ $} ] (left) {};
 	\draw[kernel1] (right) to
 	node [sloped,below] {\small }     (root);
 	\draw[kernel1] (centerl) to
 	node [sloped,below] {\small }     (left);
 	\draw[kernel1] (center) to
 	node [sloped,below] {\small }     (left); \draw[kernel1] (left) to
 	node [sloped,below] {\small }     (root);
 \end{tikzpicture} +  \begin{tikzpicture}[scale=0.2,baseline=0.1cm]
 	\node at (0,0)  [dot,label= {[label distance=-0.2em]below: \scriptsize  $   $} ] (root) {};
 	\node at (-1,4)  [dot,label= {[label distance=-0.2em]above: \scriptsize  $     $} ] (center) {};
 	\node at (-1,6)  [dot,label= {[label distance=-0.2em]above: \scriptsize  $     $} ] (centerc) {};
 	\node at (1,2)  [dot,label={[label distance=-0.2em]right: \scriptsize  $  $}] (right) {};
 	\node at (-1,2)  [dot,label={[label distance=-0.2em]above: \scriptsize  $ $} ] (left) {};
 	\draw[kernel1] (right) to
 	node [sloped,below] {\small }     (root);
 	\draw[kernel1] (center) to
 	node [sloped,below] {\small }     (centerc);
 	\draw[kernel1] (center) to
 	node [sloped,below] {\small }     (left); \draw[kernel1] (left) to
 	node [sloped,below] {\small }     (root);
 \end{tikzpicture}  + \begin{tikzpicture}[scale=0.2,baseline=0.1cm]
 \node at (0,0)  [dot,label= {[label distance=-0.2em]below: \scriptsize  $   $} ] (root) {};
 \node at (-1,4)  [dot,label= {[label distance=-0.2em]above: \scriptsize  $     $} ] (center) {};
 \node at (1,4)  [dot,label= {[label distance=-0.2em]above: \scriptsize  $     $} ] (center1) {};
 \node at (1,2)  [dot,label={[label distance=-0.2em]right: \scriptsize  $  $}] (right) {};
 \node at (-1,2)  [dot,label={[label distance=-0.2em]above: \scriptsize  $ $} ] (left) {};
 \draw[kernel1] (right) to
 node [sloped,below] {\small }     (root);
 \draw[kernel1] (center) to
 node [sloped,below] {\small }     (left); \draw[kernel1] (left) to
 node [sloped,below] {\small }     (root);
  \draw[kernel1] (center1) to
 node [sloped,below] {\small }     (right); 
\end{tikzpicture}.
 \end{equs}
If we consider the free Lie algebra generated by $ \mathcal{P} \mathcal{T} $ equipped with the left-most grafting, one obtains the free post-Lie algebra (see \cite[Thm. 3.2]{ML13}). 
\section{ Decorated trees, multi-indices and non-commuting operators}
\label{section_trees_multi_indices}

 \ \ In this section, we recall two different combinatorial structures that are both used in the context of Regularity Structures with the aim of solving singular SPDEs. The first one is that of decorated trees and the second one is that of multi-indices. We shall stress the non-commutative nature of some key operators defined on these structures and proceed to show that their non-commutative properties can be elegantly captured by certain post-Lie algebraic structures.

Decorated trees as introduced in
\cite{BHZ} are described in the following way. Pick two symbols $I$ and $\Xi$ and let $ \mathcal{D} := \lbrace I,\Xi \rbrace \times \mathbb{N}^{d+1}$ define the set of edge decorations. These two symbols represent a convolution with a kernel $ I $ and a noise term $ \Xi $. One may add more symbols if one works with a system of SPDEs with more than one noise and one kernel. Decorated trees over $ \mathcal{D} $ are of the form  $T_{\Labe}^{\Labn} =  (T,\Labn,\Labe) $ where $T$ is a non-planar rooted tree with node set $N_T$ and edge set $E_T$. The maps $\Labn : N_T \rightarrow \mathbb{N}^{d+1}$ and $\Labe : E_T \rightarrow \mathcal{D}$ are node, respectively edge, decorations. We denote the set of decorated trees by $ \mfT $. The tree product is defined by 
\begin{equation*} 
 	(T,\Labn,\Labe) \cdot  (T',\Labn',\Labe') 
 	= (T \cdot T',\Labn + \Labn', \Labe + \Labe')\;, 
\end{equation*} 
where $T \cdot T'$ is the rooted tree obtained by identifying the roots of $ T$ and $T'$. The sums $ \Labn + \Labn'$ mean that decorations are added at the root and extended to the disjoint union by setting them to vanish on the other tree.  
 We make the connection with  symbolic notation introduced in the previous part.

\begin{enumerate}
   \item[--] An edge decorated by  $ (I,a) \in \mathcal{D} $  is denoted by $ I_{a} $. The symbol $  I_{a} $ is also viewed as  the operation that grafts a tree onto a new root via a new edge with edge decoration $ a $. The new root at hand remains decorated with $0$. 
   
   \item[--]  An edge decorated by $ (\Xi,0) \in \mathcal{D} $ is denoted by $  \Xi $.  We suppose that these edges are terminal edges with zero node decoration at their leaves.

   \item[--] A factor $ X^{\ell}$   encodes a single node  $ \bullet^{\ell} $ decorated by $ \ell \in \mathbb{N}^{d+1}$. We write $ X_i$, $ i \in \lbrace 0,1,\ldots,d\rbrace $, to denote $ X^{e_i}$. Here, we have denoted by $ e_i $ the vector of $ \mathbb{N}^{d+1} $ with $ 1 $ in $i$th position and $ 0 $ otherwise. The element $ X^0 $ is identified with the empty tree $\mathbbm{1}$.
 \end{enumerate}
 Using this symbolic notation, given a decorated tree $ \tau $ there exist decorated trees $ \tau_1, ..., \tau_r $ such that
 \begin{equs}
 \tau = X^{\ell} \Xi^m  \prod_{i=1}^r I_{a_i}(\tau_i)
 \end{equs}
 where $ \prod_i $ is the tree product, $ \ell \in \mathbb{N}^{d+1} $, $ m, r \in \mathbb{N} $. In relevant applications, a product of noises is not allowed and one can only consider the case $ m \in \lbrace 0,1 \rbrace $. A tree of the form $ I_a(\tau) $ is called a planted tree as there is only one edge connecting the root to the rest of the tree.
 Below, we present an example of  decorated trees:
 \begin{equs} 
 	\tau = X^{\alpha} \Xi I_a(X^{\beta})  =   \begin{tikzpicture}[scale=0.2,baseline=0.1cm]
 		\node at (0,0)  [dot,label= {[label distance=-0.2em]below: \scriptsize  $  \alpha   $} ] (root) {};
 		\node at (2,4)  [dot,label={[label distance=-0.2em]above: \scriptsize  $ \beta $}] (right) {};
 		\node at (-2,4)  [dot,label={[label distance=-0.2em]above: \scriptsize  $ $} ] (left) {};
 		\draw[kernel1] (right) to
 		node [sloped,below] {\small }     (root); \draw[kernel1] (left) to
 		node [sloped,below] {\small }     (root);
 		\node at (-1,2) [fill=white,label={[label distance=0em]center: \scriptsize  $ \Xi $} ] () {};
 		\node at (1,2) [fill=white,label={[label distance=0em]center: \scriptsize  $ a $} ] () {};
 	\end{tikzpicture}, \quad 
     I_b(\tau) = 
 \begin{tikzpicture}[scale=0.2,baseline=0.1cm]
 	\node at (0,0)  [dot,label= {[label distance=-0.2em]below: \scriptsize  $  $} ] (root) {};
 	\node at (-3,3)  [dot,label={[label distance=-0.2em]left: \scriptsize  $ \alpha $} ] (left) {};
 	\node at (0,7)  [dot,label={[label distance=-0.2em]above: \scriptsize  $ \beta $} ] (center) {};
 	\node at (-6,7)  [dot,label={[label distance=-0.2em]above: \scriptsize  $  $} ] (centerr) {};
 	\draw[kernel1] (left) to
 	node [sloped,below] {\small }     (root);
 	\draw[kernel1] (center) to
 	node [sloped,below] {\small }     (left);
 	\draw[kernel1] (centerr) to
 	node [sloped,below] {\small }     (left);
 	\node at (-1.5,1.5) [fill=white,label={[label distance=0em]center: \scriptsize  $ b $} ] () {};
 	\node at (-1.5,5) [fill=white,label={[label distance=0em]center: \scriptsize  $ a $} ] () {};
 	\node at (-4.5,5) [fill=white,label={[label distance=0em]center: \scriptsize  $ \Xi $} ] () {};
 \end{tikzpicture}
 \end{equs}
where we have put $ \Xi $ on the edge to specify that it is decorated by $ (\Xi,0) $. For an edge decorated by $ (I,a) $, we have just put $ a $. Nodes without decoration mean that their decoration is equal to zero.  We define a product called grafting product:
\begin{equation*} 
\sigma \curvearrowright^a \tau:=\sum_{v\in  N_{\tau} } \sigma \curvearrowright^a_v  \tau,
\end{equation*}
where $\sigma $ and $\tau$ are two decorated rooted trees, $ N_\tau $ is the set of vertices of $ \tau $ and where $\sigma \curvearrowright^a_v \tau$ is obtained by grafting the tree $\sigma$ on the tree $\tau$ at vertex $v$ by means of a new edge decorated by $a\in \mathbb{N}^{d+1}$. Grafting onto noise-type edges, that is, edges decorated by $ (\Xi,0) $ is forbidden. Therefore, there is a bijection between noises as decorated edges and noises as node decorations. Below, we provide an example of this grafting product: 
\begin{equs} 
 \bullet^{\alpha}  \curvearrowright^a   \begin{tikzpicture}[scale=0.2,baseline=0.1cm]
		\node at (0,0)  [dot,label= {[label distance=-0.2em]below: \scriptsize  $  \gamma   $} ] (root) {};
		\node at (2,4)  [dot,label={[label distance=-0.2em]above: \scriptsize  $ \beta $}] (right) {};
		\node at (-2,4)  [dot,label={[label distance=-0.2em]above: \scriptsize  $ $} ] (left) {};
		\draw[kernel1] (right) to
		node [sloped,below] {\small }     (root); \draw[kernel1] (left) to
		node [sloped,below] {\small }     (root);
		\node at (-1,2) [fill=white,label={[label distance=0em]center: \scriptsize  $ \Xi $} ] () {};
		\node at (1,2) [fill=white,label={[label distance=0em]center: \scriptsize  $ b $} ] () {};
	\end{tikzpicture} = \begin{tikzpicture}[scale=0.2,baseline=0.1cm]
	\node at (0,0)  [dot,label= {[label distance=-0.2em]below: \scriptsize  $  \gamma   $} ] (root) {};
	\node at (0,5)  [dot,label= {[label distance=-0.2em]above: \scriptsize  $  \alpha   $} ] (center) {};
	\node at (3,4)  [dot,label={[label distance=-0.2em]above: \scriptsize  $ \beta $}] (right) {};
	\node at (-3,4)  [dot,label={[label distance=-0.2em]above: \scriptsize  $ $} ] (left) {};
	\draw[kernel1] (right) to
	node [sloped,below] {\small }     (root); 
	\draw[kernel1] (center) to
	node [sloped,below] {\small }     (root); 
	\draw[kernel1] (left) to
	node [sloped,below] {\small }     (root);
	\node at (-2,2) [fill=white,label={[label distance=0em]center: \scriptsize  $ \Xi $} ] () {};
	\node at (2,2) [fill=white,label={[label distance=0em]center: \scriptsize  $ b $} ] () {};
	\node at (0,2.5) [fill=white,label={[label distance=0em]center: \scriptsize  $ a $} ] () {};
\end{tikzpicture} + \begin{tikzpicture}[scale=0.2,baseline=0.1cm]
\node at (0,0)  [dot,label= {[label distance=-0.2em]below: \scriptsize  $  \gamma  $} ] (root) {};
\node at (0,8)  [dot,label= {[label distance=-0.2em]above: \scriptsize  $  \alpha   $} ] (center) {};
\node at (2,4)  [dot,label={[label distance=-0.2em]right: \scriptsize  $ \beta $}] (right) {};
\node at (-2,4)  [dot,label={[label distance=-0.2em]above: \scriptsize  $ $} ] (left) {};
\draw[kernel1] (right) to
node [sloped,below] {\small }     (root);
\draw[kernel1] (center) to
node [sloped,below] {\small }     (right); \draw[kernel1] (left) to
node [sloped,below] {\small }     (root);
\node at (-1,2) [fill=white,label={[label distance=0em]center: \scriptsize  $ \Xi $} ] () {};
\node at (1,2) [fill=white,label={[label distance=0em]center: \scriptsize  $ b $} ] () {};
\node at (1,6) [fill=white,label={[label distance=0em]center: \scriptsize  $ a $} ] () {};
\end{tikzpicture}. 
\end{equs}
One should notice that we did not graft onto the noise edge. The family of grafting products $ (\curvearrowright^b  )_{b \in \mathbb{N}^{d+1}} $ forms a multi-pre-Lie algebra, in the sense that they satisfy the following identities:
\begin{equs}
	\left( \tau_1  \curvearrowright^a \tau_2 \right)  \curvearrowright^b \tau_3 - 	 \tau_1  \curvearrowright^a (  \tau_2  \curvearrowright^b \tau_3 ) = 	( \tau_2  \curvearrowright^b \tau_1 )  \curvearrowright^a \tau_3 - 	 \tau_2  \curvearrowright^b (  \tau_1  \curvearrowright^a \tau_3 )
\end{equs}
where the $ \tau_i $ are decorated trees and $ a,b $ belong to $ \mathbb{N}^{d+1} $.
This is an extension of the notion of a pre-Lie product (which is recovered when the family is reduced to one element) and was first introduced in \cite[Prop. 4.21]{BCCH}.
This multi-pre-Lie algebra can be summarised into a single pre-Lie structure on the space of planted trees, given by the product:
\begin{equation*} 
I_{a}(\sigma) \curvearrowright I_{b}(\tau):= I_{b}(\sigma \curvearrowright^a \tau).
\end{equation*}
This was first noticed in \cite{F2018} (see also \cite[Prop. 3.2]{BM22}). Below, we provide an example that illustrates the grafting operation:

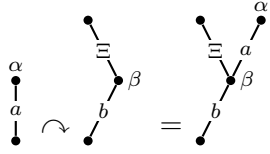
\begin{equs} 
\begin{tikzpicture}[scale=0.2,baseline=0.1cm]
	\node at (0,0)  [dot,label= {[label distance=-0.2em]below: \scriptsize  $    $} ] (root) {};
	\node at (0,4)  [dot,label={[label distance=-0.2em]above: \scriptsize  $ \alpha $}] (center) {};
	\draw[kernel1] (center) to
	node [sloped,below] {\small }     (root); 
	\node at (0,2) [fill=white,label={[label distance=0em]center: \scriptsize  $ a $} ] () {};
\end{tikzpicture} 	 \curvearrowright  \begin{tikzpicture}[scale=0.2,baseline=0.1cm]
\node at (0,0)  [dot,label= {[label distance=-0.2em]below: \scriptsize  $     $} ] (root) {};
\node at (0,8)  [dot,label= {[label distance=-0.2em]above: \scriptsize  $    $} ] (center) {};
\node at (2,4)  [dot,label={[label distance=-0.2em]right: \scriptsize  $ \beta $}] (right) {};
\draw[kernel1] (right) to
node [sloped,below] {\small }     (root);
\draw[kernel1] (center) to
node [sloped,below] {\small }     (right); 
\node at (1,2) [fill=white,label={[label distance=0em]center: \scriptsize  $ b $} ] () {};
\node at (1,6) [fill=white,label={[label distance=0em]center: \scriptsize  $ \Xi $} ] () {};
\end{tikzpicture}   =   \begin{tikzpicture}[scale=0.2,baseline=0.1cm]
\node at (0,0)  [dot,label= {[label distance=-0.2em]below: \scriptsize  $     $} ] (root) {};
\node at (0,8)  [dot,label= {[label distance=-0.2em]above: \scriptsize  $    $} ] (center) {};
\node at (4,8)  [dot,label= {[label distance=-0.2em]above: \scriptsize  $ \alpha   $} ] (centerl) {};
\node at (2,4)  [dot,label={[label distance=-0.2em]right: \scriptsize  $ \beta $}] (right) {};
\draw[kernel1] (right) to
node [sloped,below] {\small }     (root);
\draw[kernel1] (center) to
node [sloped,below] {\small }     (right); 
\draw[kernel1] (centerl) to
node [sloped,below] {\small }     (right); 
\node at (1,2) [fill=white,label={[label distance=0em]center: \scriptsize  $ b $} ] () {};
\node at (3,6) [fill=white,label={[label distance=0em]center: \scriptsize  $ a $} ] () {};
\node at (1,6) [fill=white,label={[label distance=0em]center: \scriptsize  $ \Xi $} ] () {};
\end{tikzpicture}
\end{equs}
Notice that we do not graft at the root, which has its decoration set to zero.

Furthermore, the products $ \curvearrowright^a $ can be deformed via a pre-Lie isomorphism described in \cite[Sec. 2.2]{BM22}. The deformed products are given by:

\begin{equation*} 
\sigma \widehat{\curvearrowright}^a \tau:=\sum_{v\in N_{\tau}}\sum_{\ell\in\mathbb{N}^{d+1}}{\Labn_v \choose \ell} \sigma  \curvearrowright_v^{a-\ell}(\uparrow_v^{-\ell} \tau).
\end{equation*}
\label{deformed_grafting_a}

 Intuitively, one should think of the operators $\uparrow_v^{-\ell}$ as encoding the operation of partial differentiation at a purely algebraic level. Formally, $ \Labn_v \in \mathbb{N}^{d+1}$ denotes  the decoration at the vertex $ v $ and the operator $ \uparrow_v^{-\ell} $ is defined  as subtracting $ \ell $ from the node decoration of $ v $.

 The generic term is self-explanatory if there exists a (unique) pair $(b,\alpha)\in \mathbb{N}^{d+1} \times \mathbb{N}^{d+1}$ such that $a=\ell+b$ and $\Labn_v =\ell+\alpha$. It vanishes by convention if this condition is not satisfied. Given a scaling $\s \in\mathbb{N}_0^{d+1} = \mathbb{N}^{d+1} \setminus \lbrace 0 \rbrace$ we define the \textsl{grading} of a tree as the sum of the gradings of its edges and denote it by $ |\cdot|_{\text{grad}} $:
\begin{equation*} 
|\tau|_{\text{grad}}:=\sum_{e\in E_{\tau}}\big|\Labe(e) \big|_{\s}
\end{equation*}
where $ E_{\tau} $ are the edges of $ \tau $, $ \Labe(e) $ is the decoration of the edge $ e $ and for a given $ \mathbf{n} \in \mathbb{N}^{d+1} $, one has:
 \begin{equs}
|\mathbf n|_{\s}:= \sum_{i=0}^d s_i n_i.
\end{equs}
 A  scaling $\s$ is necessary as these decorated trees are used in the context of singular SPDEs where different variables come with different degrees of homogeneity. A good example to keep in mind is that of parabolic
equations, where the linear operator is given by $ \partial_t - \Delta $, with the Laplacian acting on the spatial variable. It is then natural to make powers of t “count double" and take the parabolic scaling $ (2,1,...,1) $.

 The deformed product $\widehat{\curvearrowright}^a$ has been first introduced in \cite[Rem. 4.12]{BCCH}.
One notices that $\widehat{\curvearrowright}^a$ is a deformation of $\curvearrowright^a$ in the sense that:
\begin{equs}
\sigma \widehat{\curvearrowright}^a \tau = \sigma \curvearrowright^a \tau +
\hbox{ lower grading terms}.
\end{equs}
Again, one may summarise the above family of multi-pre-Lie products into a single pre-Lie product. More specifically, the deformed pre-Lie product $\widehat{\curvearrowright}$ is given by:
\begin{equs}
I_{a}(\sigma) \, \widehat{\curvearrowright} \, I_{b}(\tau):= I_{b}(\sigma \, \widehat{\curvearrowright}^{a} \, \tau).
\end{equs}
Below, we provide an example illustrating the product $ \widehat{\curvearrowright} $:
\begin{equs} 
	\begin{tikzpicture}[scale=0.2,baseline=0.1cm]
		\node at (0,0)  [dot,label= {[label distance=-0.2em]below: \scriptsize  $    $} ] (root) {};
		\node at (0,4)  [dot,label={[label distance=-0.2em]above: \scriptsize  $ \alpha $}] (center) {};
		\draw[kernel1] (center) to
		node [sloped,below] {\small }     (root); 
		\node at (0,2) [fill=white,label={[label distance=0em]center: \scriptsize  $ a $} ] () {};
	\end{tikzpicture} 	 \widehat{\curvearrowright} \begin{tikzpicture}[scale=0.2,baseline=0.1cm]
		\node at (0,0)  [dot,label= {[label distance=-0.2em]below: \scriptsize  $     $} ] (root) {};
		\node at (0,8)  [dot,label= {[label distance=-0.2em]above: \scriptsize  $    $} ] (center) {};
		\node at (2,4)  [dot,label={[label distance=-0.2em]right: \scriptsize  $ \beta $}] (right) {};
		\draw[kernel1] (right) to
		node [sloped,below] {\small }     (root);
		\draw[kernel1] (center) to
		node [sloped,below] {\small }     (right); 
		\node at (1,2) [fill=white,label={[label distance=0em]center: \scriptsize  $ b $} ] () {};
		\node at (1,6) [fill=white,label={[label distance=0em]center: \scriptsize  $ \Xi $} ] () {};
	\end{tikzpicture}   =   \begin{tikzpicture}[scale=0.2,baseline=0.1cm]
		\node at (0,0)  [dot,label= {[label distance=-0.2em]below: \scriptsize  $     $} ] (root) {};
		\node at (0,8)  [dot,label= {[label distance=-0.2em]above: \scriptsize  $    $} ] (center) {};
		\node at (4,8)  [dot,label= {[label distance=-0.2em]above: \scriptsize  $ \alpha   $} ] (centerl) {};
		\node at (2,4)  [dot,label={[label distance=-0.2em]right: \scriptsize  $ \beta $}] (right) {};
		\draw[kernel1] (right) to
		node [sloped,below] {\small }     (root);
		\draw[kernel1] (center) to
		node [sloped,below] {\small }     (right); 
		\draw[kernel1] (centerl) to
		node [sloped,below] {\small }     (right); 
		\node at (1,2) [fill=white,label={[label distance=0em]center: \scriptsize  $ b $} ] () {};
		\node at (3,6) [fill=white,label={[label distance=0em]center: \scriptsize  $ a $} ] () {};
		\node at (1,6) [fill=white,label={[label distance=0em]center: \scriptsize  $ \Xi $} ] () {};
	\end{tikzpicture} + \sum_{\ell \in \mathbb{N}^{d+1}_0 }{\beta \choose \ell} \, \, \begin{tikzpicture}[scale=0.2,baseline=0.1cm]
	\node at (0,0)  [dot,label= {[label distance=-0.2em]below: \scriptsize  $     $} ] (root) {};
	\node at (0,8)  [dot,label= {[label distance=-0.2em]above: \scriptsize  $    $} ] (center) {};
	\node at (4,8)  [dot,label= {[label distance=-0.2em]above: \scriptsize  $ \alpha   $} ] (centerl) {};
	\node at (2,4)  [dot,label={[label distance=-0.2em]right: \scriptsize  $ \beta - \ell $}] (right) {};
	\draw[kernel1] (right) to
	node [sloped,below] {\small }     (root);
	\draw[kernel1] (center) to
	node [sloped,below] {\small }     (right); 
	\draw[kernel1] (centerl) to
	node [sloped,below] {\small }     (right); 
	\node at (1,2) [fill=white,label={[label distance=0em]center: \scriptsize  $ b $} ] () {};
	\node at (3,6) [fill=white,label={[label distance=0em]center: \scriptsize  $ \qquad a - \ell  $} ] () {};
	\node at (1,6) [fill=white,label={[label distance=0em]center: \scriptsize  $ \Xi $} ] () {};
\end{tikzpicture}
\end{equs}
Another important operation we will need to define on the space of trees is $ \uparrow^{i} $:
\begin{equs}
\uparrow^{i} \tau  = \sum_{v \in N_{\tau}} \uparrow^{e_i}_v \tau.  
\end{equs}
Here, the insertion does not happen on top of noise-type edges. We provide below an example of computation:
\begin{equs} 
	\uparrow^{i} \begin{tikzpicture}[scale=0.2,baseline=0.1cm]
		\node at (0,0)  [dot,label= {[label distance=-0.2em]below: \scriptsize  $  \gamma  $} ] (root) {};
		\node at (2,4)  [dot,label={[label distance=-0.2em]above: \scriptsize  $ \beta $}] (right) {};
		\node at (-2,4)  [dot,label={[label distance=-0.2em]above: \scriptsize  $ $} ] (left) {};
		\draw[kernel1] (right) to
		node [sloped,below] {\small }     (root); \draw[kernel1] (left) to
		node [sloped,below] {\small }     (root);
		\node at (-1,2) [fill=white,label={[label distance=0em]center: \scriptsize  $ \Xi $} ] () {};
		\node at (1,2) [fill=white,label={[label distance=0em]center: \scriptsize  $ b $} ] () {};
	\end{tikzpicture} =  \begin{tikzpicture}[scale=0.2,baseline=0.1cm]
	\node at (0,0)  [dot,label= {[label distance=-0.2em]below: \scriptsize  $  \gamma + e_i  $} ] (root) {};
	\node at (2,4)  [dot,label={[label distance=-0.2em]above: \scriptsize  $ \beta $}] (right) {};
	\node at (-2,4)  [dot,label={[label distance=-0.2em]above: \scriptsize  $ $} ] (left) {};
	\draw[kernel1] (right) to
	node [sloped,below] {\small }     (root); \draw[kernel1] (left) to
	node [sloped,below] {\small }     (root);
	\node at (-1,2) [fill=white,label={[label distance=0em]center: \scriptsize  $ \Xi $} ] () {};
	\node at (1,2) [fill=white,label={[label distance=0em]center: \scriptsize  $ b $} ] () {};
\end{tikzpicture} +  \begin{tikzpicture}[scale=0.2,baseline=0.1cm]
\node at (0,0)  [dot,label= {[label distance=-0.2em]below: \scriptsize  $  \gamma  $} ] (root) {};
\node at (2,4)  [dot,label={[label distance=-0.2em]above: \scriptsize  $ \beta + e_i $}] (right) {};
\node at (-2,4)  [dot,label={[label distance=-0.2em]above: \scriptsize  $ $} ] (left) {};
\draw[kernel1] (right) to
node [sloped,below] {\small }     (root); \draw[kernel1] (left) to
node [sloped,below] {\small }     (root);
\node at (-1,2) [fill=white,label={[label distance=0em]center: \scriptsize  $ \Xi $} ] () {};
\node at (1,2) [fill=white,label={[label distance=0em]center: \scriptsize  $ b $} ] () {};
\end{tikzpicture}.
\end{equs}
This operator is a derivation for the grafting  product $ \curvearrowright^a $ in the sense that
\begin{equs}
 \uparrow^{i} \left( \sigma \curvearrowright^a \tau  \right) =  (\uparrow^{i} \sigma) \curvearrowright^a  \tau +  \sigma \curvearrowright^a \, ( \uparrow^{i} \tau)
\end{equs}
and one has the following right derivation property:
\begin{equs}
\uparrow^{i}_{N_\tau} \left( \sigma \curvearrowright^a \tau  \right) =    \sigma \curvearrowright^a \, ( \uparrow^{i} \tau)
\end{equs}
where $\uparrow^{i}_{N_\tau} $ is defined as $\uparrow^{i}  $ but with $ N_{\tau} \sqcup N_{\sigma} $ replaced by $ N_{\tau} $. These two properties are clearly not true for the deformed pre-Lie products $ \widehat{\curvearrowright}^{a}  $. They are, however, true up to some deformation. This is made formal by the following proposition:

\begin{proposition} \label{prop_non_com}
One has, for all decorated trees $ \sigma, \tau $ and $ a \in \mathbb{N}^{d+1} $, $ i \in \lbrace 0,...,d \rbrace $:
\begin{equs} \label{non_commutation}
\uparrow^{i}_{N_\tau} \left( \sigma \widehat{\curvearrowright}^a \tau  \right) =    \sigma \widehat{\curvearrowright}^a \, ( \uparrow^{i} \tau)  - \sigma \widehat{\curvearrowright}^{a-e_i} \tau 
\end{equs}
\end{proposition}
\begin{proof}
 Identity \eqref{non_commutation} is a consequence of the following:
 \begin{equs}
  \sigma \widehat{\curvearrowright}^a \, ( \uparrow^{i} \tau) & =
    \sum_{v\in N_{\tau}}\sum_{\ell\in\mathbb{N}^{d+1}}{\Labn_v + e_i \choose \ell} \sigma  \curvearrowright_v^{a-\ell}(\uparrow_v^{-\ell + e_i} \tau)
    \\ \sigma \widehat{\curvearrowright}^{a-e_i} \,  \tau & = 
    \sum_{v\in N_{\tau}}\sum_{\ell\in\mathbb{N}^{d+1}}{\Labn_v  \choose \ell} \sigma  \curvearrowright_v^{a-\ell - e_i}(\uparrow_v^{-\ell } \tau)
    \\ & =  \sum_{v\in N_{\tau}}\sum_{\ell\in\mathbb{N}^{d+1}}{\Labn_v  \choose \ell-e_i} \sigma  \curvearrowright_v^{a-\ell }(\uparrow_v^{-\ell + e_i } \tau) \\
    \uparrow^{i}_{N_\tau} \left( \sigma \widehat{\curvearrowright}^a \tau  \right) & = \sum_{v\in N_{\tau}}\sum_{\ell\in\mathbb{N}^{d+1}}{\Labn_v  \choose \ell} \sigma  \curvearrowright_v^{a-\ell }(\uparrow_v^{-\ell + e_i } \tau).
 \end{equs}
Then, we conclude by the fact that:
\begin{equs}
{\Labn_v  \choose \ell-e_i} + {\Labn_v  \choose \ell} = {\Labn_v + e_i  \choose \ell}.
\end{equs}
\end{proof}
\begin{remark}
The extra term in \eqref{non_commutation} can be seen as a term of lower order as the decorations of the grafting operator are decreased by $e_{i}$.
\end{remark}

Recently, a different encoding of iterated integrals has been introduced in \cite{LOT} based on multi-indices. Let's briefly recall the definitions given by the authors. We suppose we are given two sets of abstract variables $ (z_k)_{k \in \mathbb{N}} $ and $ (z_{n})_{n \in \mathbb{N}^{d+1}} $. The $ z_k $ encode nodes of arity $ k $ and the $ z_n $ are monomials. Multi-indices $ \beta  $ over $ \mathbb{N} $ and $ \mathbb{N}^{d+1} $ measure the frequency of the variables $ z_k, z_n $, so that we
can write monomials
\begin{equs}
z^{\beta} : = \prod_{k \in \mathbb{N}, \, n \in \mathbb{N}^{d+1}} z_k^{\beta(k)} z_{n}^{\beta(n)}. 
\end{equs}
One can write $ \beta $ according  to the  canonical basis of $ \bar{e}_k, \bar{e}_n $ encoding the $ z_k, z_{n} $:
\begin{equs}
\beta = \sum_{k \in \mathbb{N}} \beta(k) \bar{e}_k + \sum_{n \in \mathbb{N}^{d+1}} \beta(n) \bar{e}_n.
\end{equs}
We introduce a family of derivations on these multi-indices: $ D^{(0)}, (D^{(n)})_{n \neq 0} $ and $ \partial_i,  \,  i \in \lbrace 0,...,d\rbrace $.  These are defined in the following way:
\begin{equs}
D^{(0)} = \sum_{k \in \mathbb{N}} (k+1) z_{k+1} \partial_{z_k},  \quad D^{(n)} = \partial_{z_n}, \, n \neq 0,
\end{equs}
where $ \partial_{z_k} $ is the derivative in the coordinates $ z_k $. The action of this derivative operator corresponds to increasing the arity of a node by one.
Then, the derivatives $ \partial_i $ are given by:
\begin{equs}
\partial_i  = \sum_{n} (n_i +1 ) z_{n + e_i} D^{(n)}.
\end{equs}
For all these derivatives, we will use a matrix representation $ (D^{(0)})^{\gamma}_{\beta}, (D^{(n)})^{\gamma}_{\beta} $ and $ (\partial_i)^{\gamma}_\beta$ where $ \gamma $ and $ \beta  $ are multi-indices.
\begin{align*}
	(D^{({ 0})})_\beta^\gamma=\sum_{k\ge 0}\left\{\begin{array}{cl}
		(k+1)\gamma(k)&\mbox{if}\;\gamma+\bar{e}_{k+1}=\beta+\bar{e}_k\\
		0&\mbox{otherwise}\end{array}\right\}
\end{align*}
\begin{align*}
	(D^{({ n})})_\beta^\gamma=\left\{\begin{array}{cl}
		\gamma({ n})&\mbox{if}\;\gamma=\beta+\bar{e}_{ n}\\
		0&\mbox{otherwise}\end{array}\right\}\quad\mbox{for}\;{ n}\not={ 0}
\end{align*}
\begin{equs} 
	\begin{aligned}
		(\partial_i)_\beta^\gamma&=\sum_{k\ge 0}\left\{\begin{array}{cl}
			(k+1)\gamma(k)&\mbox{if}\;\gamma+\bar{e}_{k+1}+\bar{e}_{e_i}=\beta+e_k \\
			0&\mbox{otherwise}\end{array}\right\} \\
		&+\sum_{{ n}\not=0}\left\{\begin{array}{cl}
			(n_i+1)\gamma({ n})&\mbox{if}\;\gamma+\bar{e}_{{ n}+ e_i}=\beta+\bar{e}_{ n}\\
			0&\mbox{otherwise}\end{array}\right\}.
	\end{aligned}
\end{equs}
The non-commutative property  \eqref{non_commutation}
 appears at the level of the derivations. Indeed, one has:
\begin{equs} \label{non-commutation_2}
\partial_i D^{(n)} = D^{(n)} \partial_i + n_i D^{(n-e_i)}.
\end{equs}
This motivates the introduction of a Lie bracket in the next section for taking into account that non-commutative property. We note that the encoding given by multi-indices does not precisely correspond to the original encoding via decorated trees. Indeed, the decorated trees associated to them contain more node decorations, in the spirit of \cite{BCCH} where a new class of trees has been introduced for proving a pre-Lie morphism property on some elementary differentials. The new decorated trees are of this form:
\begin{equs} \label{trees_edges_no_decoration}
\tau =  \Xi \prod_{j=1}^{k_1} X^{\ell_j}  \prod_{i=1}^{k_2} I(\tau_i)
\end{equs}
where now the noise $ \Xi $ systematically appears at every node and can be omitted in the notation. All the edges have the same decoration $ (I,0) $. The main novelty are the new decorations given by the product of the $ X^{m_j} $ and this time, we have:
\begin{equs}
\prod_{j=1}^{k_1} X^{\ell_j} \neq X^{\sum_{j=1}^{k_1} \ell_j}. 
\end{equs}
We denote the set of these decorated trees by $ \mathfrak{T}_0 $. 
Then, the mapping to a multi-index can be performed recursively via the following map $ \Psi $:
\begin{equs} \label{pre_Lie_isomorphism}
\Psi(\tau) = (k_1 + k_2)! \, z_{k_1 + k_2}  \prod_{j=1}^{k_1 } (\ell_j !) \, z_{\ell_j} \prod_{i=1}^{k_2} \Psi(\tau_i), \quad \Psi(\Xi) = \one.
\end{equs}
The deformed pre-Lie product in this context takes the form:
\begin{equs} \label{pre_lie_bis}
	\begin{aligned}
\sigma \widehat{\curvearrowright}^a \tau & :=\sum_{v\in N_{\tau}}\sum_{\ell\in \lbrace 0, a \rbrace}{\Labn_v \choose \ell} \sigma  \widehat{\curvearrowright}_v^{a-\ell}(\hat{\uparrow}_v^{-\ell} \tau) \\ & = \sigma \curvearrowright^a \tau +  \sum_{v\in N_{\tau}} \sigma  \widehat{\curvearrowright}_v^{0}(\hat{\uparrow}_v^{-a} \tau)
\end{aligned}
\end{equs}
where $ \hat{\uparrow}_v^{-a} $ removes one $ X^{a} $ at the node $ v $ otherwise it is equal to zero. Then, we need to put an extra restriction as decorated trees described in \eqref{trees_edges_no_decoration} have only zero edge-decorations. Therefore, when $ a = 0 $, one has
\begin{equs}
	\sigma \widehat{\curvearrowright}^a \tau = \sigma \curvearrowright^a \tau.
\end{equs}
Otherwise, one has
\begin{equs}
	\sigma \widehat{\curvearrowright}^a \tau = \sum_{v\in N_{\tau}} \sigma  \widehat{\curvearrowright}_v^{0}(\hat{\uparrow}_v^{-a} \tau).
\end{equs}
 Such deformation is reminiscent of the one used in the context of numerical analysis in \cite{BS,BM22}, where there is only one term of lower order. 
In \cite{OSSW,LOT}, the authors consider only a subclass of multi-indices, that is
\begin{equs}
\{ z^{\gamma}D^{(n)} \}_{[\gamma] \geq 0 }
\end{equs}
where $ [\gamma] $ is given by:
\begin{equation*}
[\gamma] = \sum_{k \geq 0} k \gamma(k) - \sum_{n \neq 0} \gamma(n). 
\end{equation*}
The condition $ [\gamma] \geq 0 $ corresponds to the fact that we have fewer monomial multi-indices than branches coming from the $ \bar{e}_k $. The authors are also more restrictive by projecting according to the homogeneity of the multi-indices given by a map $ \gamma \mapsto |\gamma|_{\s} $ depending on the chosen scaling $ \s $. The extra condition is given by: $ |\gamma|_{\s} - |n|_{\s} > 0 $.
In the context of decorated trees, it corresponds to planted trees with positive homogeneity.

\section{ A post-Lie algebra for decorated trees }

In this section, we explain how the previous formalism can be expressed directly on decorated trees in order to obtain the coproduct $\Delta_{2}$ appearing in \cite{BM22} which has been introduced in \cite{reg,BHZ}. This coproduct is crucial for defining the Hopf algebra governing the recentering procedure in Regularity Structures. It is used for constructing recentered iterated integrals that are the building blocks of solutions for singular SPDEs. 

 We give analogues to the spaces introduced in the previous section by using the letter $\mathcal{V}$. We define the following spaces:
\begin{equs}
\mathcal{V} & = \Big \langle  \{ I_a(\tau), \, a \in \mathbb{N}^{d+1}, \, \tau \in \mathfrak{T} \} \cup \{ X_{i} \}_{i = 0,..., d} \Big \rangle_{\mathbb{R}}, \\
\tilde{\mathcal{V}} & = \Big  \langle \{ I_a(\tau), \, a \in \mathbb{N}^{d+1}, \, \tau \in \mathfrak{T} \} \Big \rangle_{\mathbb{R}}.
\end{equs}
The  space $ \tilde{\mathcal{V}} $ is the linear span of planted decorated trees and $ \mathcal{V}$ is the linear span of planted trees with the monomials $ X_i $. 
We introduce a Lie bracket and a product on the space $\mathcal{V}$ that are compatible with one another and will give us a post-Lie algebraic structure.

\begin{definition} We define a product $ \widehat{\triangleright} $ on $ \mathcal{V} $ for every $ a,b \in \mathbb{N}^{d+1}, \, i,j \in  \lbrace 0,\ldots,d\rbrace $ as:
\begin{equs} \label{axiom_post}
X_i \, \widehat{\triangleright}  \,  I_{a}(\tau) = \uparrow^i I_{a}(\tau), \quad I_{a}(\tau) \,  \widehat{\triangleright}  \, X_{i} = 0, \quad  X_i \, \widehat{\triangleright}  \, X_{j} = 0,
\end{equs}
and 
\begin{equs}
I_{a}(\sigma) \, \widehat{\triangleright}  \, I_{b}(\tau) = I_{a}(\sigma) \, \widehat{\curvearrowright} \,I_{b}(\tau).
\end{equs}
\end{definition}

\begin{definition} \label{def_lie_trees}
We define the Lie bracket on $\mathcal{V}$ as $[x, y]_{0} = 0$ for $x, y \in \tilde{\mathcal{V}}$, $[x, y]_{0} = 0$ for $x, y \in \langle \ X_{i} \ \rangle_{\mathbb{R}}$ and as
\begin{equation} \label{Lie-bracket}
[I_a(\tau),X_i]_0 =  I_{a-e_i}(\tau). 
\end{equation}
\end{definition}

\begin{remark}
Note that the image of this bracket lies inside $\tilde{\mathcal{V}}$. Moreover, the definition of the Lie bracket is similar to the one on multi-indices.
\end{remark}

\begin{theorem}
The triple $( \mathcal{V}, [.,.]_{0}, \widehat{\triangleright})$ is a post-Lie algebra.
\end{theorem}

\begin{proof}
One has to check:
\begin{equs} \label{ident1_b}
x \, \widehat{\triangleright} \, [y,z]_0 = [x  \, \widehat{\triangleright} \, y,z]_0 + [y, x  \, \widehat{\triangleright} \, z]_0
\end{equs}
and
\begin{equs} \label{ident2_b}
[x,y]_0 \, \widehat{\triangleright} \, z = a_{\widehat{\triangleright}}(x,y,z) - a_{\widehat{\triangleright}}(y,x,z).
\end{equs}
It is easy to check \eqref{ident1_b}, for $ x= X_i $. Then, if we consider $ x = I_{a}(\sigma) $ by symmetry we can restrict ourselves to the case $ y = I_a(\sigma) $ and $ z= X_i $ which will be non zero. We have
\begin{equs}
I_a(\sigma) \, \widehat{\triangleright} \, [I_b(\tau),X_i]_0 = I_a(\sigma) \, \widehat{\curvearrowright} \, I_{b-e_i}(\tau) .
\end{equs}
We conclude by the fact that:
\begin{equation*}
 [ I_a(\sigma) \, \widehat{\triangleright} \, I_b(\tau),X_i ]_0 = I_{b-e_i}( I_a(\sigma) \, \widehat{\curvearrowright} \tau) = I_a(\sigma) \, \widehat{\curvearrowright} \, I_{b-e_i}(\tau).
\end{equation*}
It remains to show \eqref{ident2_b}. If $ z =X_i $, then it is zero on both sides. Let us consider $ z = I_{b}(\tau) $. If $ x $ and $ y $ are both planted trees, the fact that $ \widehat{\curvearrowright} $ is  a pre-Lie product gives the answer. For symmetry reason, we can consider $ x = I_a(\sigma) $ and $ y = X_i $.
From the left hand side, we get:
\begin{equation*}
[x,y]_0 \, \widehat{\triangleright} \, z = I_{a-e_i}(\sigma) \, \widehat{\curvearrowright} \, I_b(\tau).
\end{equation*}
From the right hand side, only two terms remain:
\begin{equation*}
\begin{aligned}
a_{\widehat{\triangleright}}(x,y,z)  & = I_a(\sigma) \, \widehat{\triangleright} \, ( X_i \, \widehat{\triangleright} \, I_b(\tau)) - (I_a(\sigma) \, \widehat{\triangleright} \, X_i) \, \widehat{\triangleright} \, I_b(\tau) 
\\ & =  I_{a}(\sigma) \, \widehat{\curvearrowright} \, I_b( \uparrow^i \tau)
\end{aligned}
\end{equation*}
because $ I_a(\sigma) \, \widehat{\triangleright} \, X_i =0  $. Then
\begin{equation*}
\begin{aligned}
 a_{\widehat{\triangleright}}(y,x,z) & = X_i \, \widehat{\triangleright} \, (I_a(\sigma)  \, \widehat{\triangleright} \, I_b(\tau)) - (X_i \, \widehat{\triangleright} \, I_a(\sigma)) \, \widehat{\triangleright} \, I_b(\tau) \\
 & = \ \uparrow^{i}(I_a(\sigma) \, \widehat{\curvearrowright} \, I_b(\tau) ) - (\uparrow^{i}I_{a}(\sigma)) \, \widehat{\curvearrowright} \, I_{b}(\tau) 
 \\ &= I_b( \uparrow^i_{N_{\tau}} (I_{a}(\sigma) \, \widehat{\curvearrowright} \,  \tau))
 \end{aligned}
\end{equation*}
where we have used 
\begin{equs}
 I_{b} ( \uparrow^{i} (I_{a}(\sigma) \, \widehat{\curvearrowright} \, \tau)) - I_{b}(I_{a}(\uparrow^{i} \sigma) \, \widehat{\curvearrowright} \, \tau))
 = I_b( \uparrow^i_{N_{\tau}} (I_{a}(\sigma) \, \widehat{\curvearrowright} \,  \tau)).
\end{equs}
Indeed by definition, one has:
\begin{equs}
\uparrow^{i} (I_{a}(\sigma) \, \widehat{\curvearrowright} \, \tau) & = \sum_{v \in N_{\sigma} \sqcup N_{\tau}}  \uparrow^{i}_v (I_{a}(\sigma) \, \widehat{\curvearrowright} \, \tau) 
\\ & = \sum_{v \in N_{\sigma} }  \uparrow^{i}_v (I_{a}(\sigma) \, \widehat{\curvearrowright} \, \tau) + \sum_{v \in  N_{\tau}}  \uparrow^{i}_v (I_{a}(\sigma) \, \widehat{\curvearrowright} \, \tau)
\\ & = (\uparrow^{i} I_{a}( \sigma)) \, \widehat{\curvearrowright} \, \tau +  \uparrow^{i}_{N_{\tau}} (I_{a}(\sigma) \, \widehat{\curvearrowright} \, \tau).
\end{equs}
 In the end, we get
\begin{equation*}
a_{\widehat{\triangleright}}(x,y,z)  - a_{\widehat{\triangleright}}(y,x,z) =  I_{a}(\tau) \, \widehat{\curvearrowright} \, I_b( \uparrow^i \sigma) - I_b( \uparrow^i_{N_{\sigma}} ( I_{a}(\tau) \, \widehat{\curvearrowright} \,  \sigma )).
\end{equation*}
The equality between the two expressions is given by Proposition~\ref{prop_non_com}.
\end{proof}

\begin{corollary}
The bracket $ [[.,.]] $ defined by 
 $$ [[x, y]] = [x, y]_{0} + x \, \widehat{\triangleright} \, y - y \, \widehat{\triangleright} \, x
 $$
 for every $ x, y \in \mathcal{V} $, is a Lie bracket on $\mathcal{V}$.
\end{corollary}

\begin{remark} The relation between the two Lie algebras $  [[.,.]]  $ and $ [., .]_{0} $ is central to the notion of post-Lie algebra. For instance in \cite[Def. 2.1]{BD}, the authors introduce post-Lie algebras by starting with the two Lie algebras. The post-Lie product then has to check several compatibility conditions with these two Lie algebras.
	\end{remark}

We denoted by $ U(\mathcal{V}_0) $ the enveloping algebra with the Lie bracket $ [.,.]_0 $ and  by $  U(\mathcal{V}) $ the enveloping algebra with the Lie bracket $ [[.,.]] $. We also set $ * $ to be the product obtained by the Guin-Oudom type procedure given in Section~\ref{section::Post_Lie}.  
\begin{theorem} \label{main_result_trees}
The Hopf algebra $U(\mathcal{V})$ is isomorphic to the Hopf algebra $(U(\mathcal{V}_0), *, \Delta)$.
\end{theorem}
\begin{proof}
This is a direct application of Theorem~\ref{main_theorem_section_2}.
\end{proof}

Following \cite{LOT}, one makes use of this isomorphism and the fact that the bracket $[.,.]_{0}$ vanishes on $\tilde{\mathcal{V}}$ to obtain the deformed coproduct given in \cite{reg,BHZ,BM22}, by selecting a basis for $U(\mathcal{V})$ that is symmetric with respect to the elements of the basis of $\tilde{\mathcal{V}}$. More precisely, along with the aid of the Poincare-Birkhoff-Witt theorem and after choosing to order the $X_{i}$  according to their indices, one obtains a basis of the form:
\begin{equs}
B_{(\textbf{F}, \textbf{m})} = \
\prod_{i = 0, ..., d} X_{i}^{m_{i}}  I_{a_{1}}(\sigma_{1}) \cdot \cdot \cdot I_{a_{n}}(\sigma_{n}) 
\end{equs}
where $\textbf{F} = I_{a_{1}}(\sigma_{1}) \cdot \cdot \cdot I_{a_{n}}(\sigma_n) $ ranges over all forests of planted trees and $\textbf{m} \in \mathbb{N}^{d+1}$. 
\begin{remark}
By virtue of Proposition~\ref{rep_datum} in Section~\ref{section::Post_Lie}, one obtains the following representation of $\mathfrak{\bar{g}}$:
$$
\rho : \mathfrak{\bar{g}} \rightarrow End_{\mathbb{R}}(U(\mathfrak{g}))
$$
$$
I_{a}(\sigma) \mapsto ( A \mapsto \sigma \widehat{\curvearrowright}^{a} A + I_{a}(\sigma)\cdot A ), \quad
X_{i} \mapsto ( A \mapsto \uparrow^{i} A + X_{i} \cdot A )
$$
where $A = B_{(\textbf{F}, \textbf{m})} = \
\prod_{i = 0, ..., d} X_{i}^{m_{i}}  I_{a_{1}}(\sigma_{1}) \cdot \cdot \cdot I_{a_{n}}(\sigma_{n}) $ is a basis element as above and $\cdot$ denotes the associative product in $U(\mathfrak{g})$. We note that, on the second line, in order for us to be able to apply the grafting product, an element $A$ can clearly be interpreted as a rooted tree by merging all the $I_{a_{i}}(\sigma_{i})$ on a single root and adjoining the appropriate label on that root.

\end{remark}

As one can play around using the commutation relations on the universal envelope $U(\mathfrak{g})$, the linear representation given in the remark above gives us hints about the resulting $*$ product on the universal enveloping algebra. Indeed, we shall now proceed to prove that the product $*$ thus obtained is identical to the product $\star_{2}$ that governs the combinatorics of recentering in regularity structures. The product $\star_{2}$ is the dual of the $\Delta_{2}$ coproduct as was proven in \cite{BM22}. Therefore, our Hopf algebra will be proven to be precisely the dual of the Hopf algebra $\CH_{2}$ used in regularity structures for handling the combinatorics of recentering. For this we shall use the following proposition that characterizes the $\star_{2}$ product and was proven in \cite[Prop. 3.17]{BM22}:

\begin{proposition}
Let $\sigma =  X^{k}  \prod_i I_{a_{i}}(\sigma_{i}) $. Then one has:
\begin{equs}
I_{b}(\sigma \star_{2} \tau) = \tilde{\uparrow}^{k}_{N_{\tau}} \( \prod_i I_{a_{i}}(\sigma_{i}) \, \widehat{\curvearrowright} \, I_{b}(\tau) \)
\end{equs}
where $ \tilde{\uparrow}^{k}_{N_{ \tau}} $ is defined by
\begin{equs} \label{splitting_polynomials}
	\tilde{\uparrow}^{k}_{N_{\tau}} =
	\sum_{k = \sum_{v \in N_{\tau}} k_v} \uparrow^{k_v}_{v}.
\end{equs}
\end{proposition}

\begin{theorem} \label{identification_star_star_2}
The product $\star_{2}$ coincides with the product $*$ obtained by the previous construction by unfolding the post-Lie operation onto the universal envelope $U(\mathcal{V}_0)$. One has for every $\sigma =  X^{k}  \prod_i I_{a_{i}}(\sigma_{i}) $
\begin{equs}
	I_{b}(\sigma \star_{2} \tau) = \sigma  \, \widehat{\triangleright} \, I_{b}(  \tau). 
\end{equs}
\end{theorem}
\begin{proof}
 When $ \sigma = X_i $ or $ \sigma = I_a(\hat{\sigma}) $, this is just a consequence of the definition of $ \widehat{\triangleright} $. If $\sigma = X_i I_{a}(\hat{\sigma})$, we make use of Proposition~\ref{prop_non_com} (here $ \tilde{\uparrow}^{e_i} $ = $ \uparrow^{i} $):
\begin{equs}
I_{b}(\sigma \star_{2} \tau) & = \uparrow^{i}_{N_{\tau}} \( I_{a}(\hat{\sigma} ) \, \widehat{\curvearrowright} \, I_{b}(\tau) \)
\\ & =   I_{a}(\hat{\sigma} ) \, \widehat{\curvearrowright} \, I_{b}( \uparrow^{i} \tau) 
-    I_{a-e_i}(\hat{\sigma} ) \, \widehat{\curvearrowright} \, I_{b}(\tau) 
\\ & = I_{a}(\hat{\sigma} ) \, \widehat{\triangleright} \, \( X_i \, \widehat{\triangleright} \, I_{b}(  \tau) \)
-    I_{a-e_i}(\hat{\sigma} ) \, \widehat{\triangleright} \, I_{b}(\tau) 
\\ & = I_{a}(\hat{\sigma} )  X_i \, \widehat{\triangleright} \, I_{b}(  \tau) - \( I_{a}(\hat{\sigma} ) \, \widehat{\triangleright} \,  X_i \) \, \widehat{\triangleright} \, I_{b}(  \tau)
-    I_{a-e_i}(\hat{\sigma} ) \, \widehat{\triangleright} \, I_{b}(\tau)
\\ & = I_{a}(\hat{\sigma} )  X_i \, \widehat{\triangleright} \, I_{b}(  \tau)
-    I_{a-e_i}(\hat{\sigma} ) \, \widehat{\triangleright} \, I_{b}(\tau)
\\ & = X_i I_{a}(\hat{\sigma} )  \, \widehat{\triangleright} \, I_{b}(  \tau)
\end{equs}
where we have used the definition of $ \widehat{\triangleright} $ in the fourth (see the recursive definition in section $2 $) and fifth lines above.
We can then proceed by induction in order to prove it for $ \sigma = X^k \prod_i I_{a_i}(\sigma_i) $ following the recursive definition of the product giving in Section~\ref{section::Post_Lie}.
\end{proof} 
	One may wonder how the post-Lie algebra introduced here relates to the post-Lie algebra on planar rooted trees. In fact, our post-Lie algebra can be interpreted as a quotient of the free object defined in Section~\ref{section::Post_Lie}. Indeed, we can consider the $ X_i $ as terminal edges. Then, we look at the set recursively defined by:
	\begin{equs}
		\mathfrak{T}_p = \Big\{ \Xi \prod_{i} A_i, \, A_i \in  \{ I_a(\tau), \, \tau \in \mathfrak{T}_p, \, a \in \mathbb{N}^{d+1} \}  \cup \{ X_{i} \}_{i = 0,..., d} \Big\}.
	\end{equs} 
 Here the product $ \Pi_i $ is not commutative and therefore $ \mathfrak{T}_p  $ contains planar decorated trees. Below, we provide an example of such decorated trees:
 
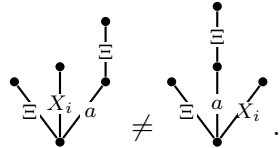
\begin{equs}
 	\begin{tikzpicture}[scale=0.2,baseline=0.1cm]
 		\node at (0,0)  [dot,label= {[label distance=-0.2em]below: \scriptsize  $     $} ] (root) {};
 		\node at (0,5)  [dot,label= {[label distance=-0.2em]above: \scriptsize  $     $} ] (center) {};
 		\node at (3,4)  [dot,label={[label distance=-0.2em]above: \scriptsize  $  $}] (right) {};
 		\node at (3,8)  [dot,label={[label distance=-0.2em]above: \scriptsize  $  $}] (rightc) {};
 		\node at (-3,4)  [dot,label={[label distance=-0.2em]above: \scriptsize  $ $} ] (left) {};
 			\draw[kernel1] (rightc) to
 		node [sloped,below] {\small }     (right);
 		\draw[kernel1] (right) to
 		node [sloped,below] {\small }     (root); 
 		\draw[kernel1] (center) to
 		node [sloped,below] {\small }     (root); 
 		\draw[kernel1] (left) to
 		node [sloped,below] {\small }     (root);
 		\node at (-2,2) [fill=white,label={[label distance=0em]center: \scriptsize  $ \Xi $} ] () {};
 		\node at (3,6) [fill=white,label={[label distance=0em]center: \scriptsize  $ \Xi $} ] () {};
 		\node at (2,2) [fill=white,label={[label distance=0em]center: \scriptsize  $ a $} ] () {};
 		\node at (0,2.5) [fill=white,label={[label distance=0em]center: \scriptsize  $ X_i $} ] () {};
 	\end{tikzpicture} \neq 	\begin{tikzpicture}[scale=0.2,baseline=0.1cm]
 	\node at (0,0)  [dot,label= {[label distance=-0.2em]below: \scriptsize  $     $} ] (root) {};
 	\node at (0,5)  [dot,label= {[label distance=-0.2em]above: \scriptsize  $     $} ] (center) {};
 	\node at (3,4)  [dot,label={[label distance=-0.2em]above: \scriptsize  $  $}] (right) {};
 	\node at (0,9)  [dot,label={[label distance=-0.2em]above: \scriptsize  $  $}] (rightc) {};
 	\node at (-3,4)  [dot,label={[label distance=-0.2em]above: \scriptsize  $ $} ] (left) {};
 	\draw[kernel1] (right) to
 	node [sloped,below] {\small }     (root); 
 		\draw[kernel1] (center) to
 	node [sloped,below] {\small }     (rightc); 
 	\draw[kernel1] (center) to
 	node [sloped,below] {\small }     (root); 
 	\draw[kernel1] (left) to
 	node [sloped,below] {\small }     (root);
 	\node at (-2,2) [fill=white,label={[label distance=0em]center: \scriptsize  $ \Xi $} ] () {};
 	\node at (2,2) [fill=white,label={[label distance=0em]center: \scriptsize  $ X_i $} ] () {};
 	\node at (0,7) [fill=white,label={[label distance=0em]center: \scriptsize  $ \Xi $} ] () {};
 	\node at (0,2.5) [fill=white,label={[label distance=0em]center: \scriptsize  $ a $} ] () {};
 \end{tikzpicture}.
 \end{equs}
Notice that now there is no decorations on the nodes. Edges decorated by $ X_i $ represent the symbol $ X_i $.
  We define new spaces $ 	\mathcal{V}_p $ and $ 	\tilde{\mathcal{V}}_p $
\begin{equs}
	\mathcal{V}_p & = \Big \langle  \{ I_a(\tau), \, a \in \mathbb{N}^{d+1}, \, \tau \in \mathfrak{T}_p \} \cup \{ X_{i} \}_{i = 0,..., d} \Big \rangle_{\mathbb{R}}, \\
	\tilde{\mathcal{V}}_p & = \Big  \langle \{ I_a(\tau), \, a \in \mathbb{N}^{d+1}, \, \tau \in \mathfrak{T}_p \} \Big \rangle_{\mathbb{R}}.
\end{equs}
One can define the left-most grafing product on these trees denoted by $ \widehat{\triangleright}_l $. One grafts on the left-most spot at the right of the noise $\Xi$. We impose one restriction which is that one cannot graft on top of the $X_i$. This gives
\begin{equs} \label{axioms_left}
	 I_{a}(\tau) \,  \widehat{\triangleright}_l  \, X_{i} = 0, \quad  X_i \, \widehat{\triangleright}_l  \, X_{j} = 0.
\end{equs}
Then, for recovering our post-Lie product, we need to quotient by certain relations that are  given by the Lie bracket \eqref{Lie-bracket} 
\begin{equation}  \label{relation}
	\begin{aligned}
 X_i X_j & = X_j X_i, \quad   I_a(\tau) I_b(\sigma) =I_b(\sigma) I_a(\tau),
\\ I_a(\tau)X_i & = X_i I_a(\tau) +  I_{a-e_i}(\tau).
\end{aligned}
\end{equation}
Using these relations, we work in the basis described by $ \mathfrak{T} $ where now the $ X_{i} $ appear before the planted trees. We shall slightly abuse notation by still using the term left-most grafting for the resulting operation after taking the quotient and will still denote it by $ \widehat{\triangleright}_l $. Then, one has that the left-most grafting product coincides with the post-Lie product $\widehat{\triangleright}$.

\begin{proposition} \label{prop_left_grafting}
	The left-most grafting $\widehat{\triangleright}_l$ coincides with $\widehat{\triangleright}$ on $ \mathcal{V} $.
	\end{proposition}
\begin{proof}
	We proceed by induction, using the basis of $\mathcal{V}$ and making use of the recursive construction of $\mathfrak{T}$ afforded to us by the sets $\mathfrak{T}_{p}$ defined above. We want to show that for every $\tau$ and $\sigma$ in $\mathcal{V}$, one has 
	\begin{equs}
		\sigma \, \widehat{\triangleright}_l \, \tau = \sigma \, \widehat{\triangleright} \, \tau. 
	\end{equs}
When $ \tau = X_i $, this is just a consequence of \eqref{axioms_left} and \eqref{axiom_post}. Now, we suppose that $ \tau = I_b(\hat{\tau}) $ where
\begin{equs}
	\hat{\tau} =  \Xi  X^{\ell} \prod_{j=1}^r I_{a_j}(\tau_j)
\end{equs}
Then, one has
\begin{equs} \label{main_1}
	\sigma \, \widehat{\triangleright}_l \, \tau = 
 I_b\left(	\Xi \sigma  X^{\ell} \prod_{j=1}^r I_{a_j}(\tau_j) + \Xi X^{\ell} \sum_{n=1}^r \left( \sigma \, \widehat{\triangleright}_l \, I_{a_n}(\tau_n) \right)  \prod_{j \neq n}  I_{a_j}(\tau_j) \right) 
	\end{equs}
which is the recursive defintion of $ \widehat{\triangleright}_l $.
We apply the induction hypothesis to conclude on the fact that for every $ n \in \{  1,...,r \} $
\begin{equs}
	 \sigma \, \widehat{\triangleright}_l \, I_{a_n}(\tau_n) = 
	  \sigma \, \widehat{\triangleright} \, I_{a_n}(\tau_n). 
\end{equs}
For the first term on the right hand side of \eqref{main_1}, one has if $ \sigma = X_i $
\begin{equs}
	\Xi \sigma  X^{\ell} \prod_{j=1}^r I_{a_j}(\tau_j) 
	& = \Xi X_i  X^{\ell} \prod_{j=1}^r I_{a_j}(\tau_j)
	\\
	& =  \Xi   X^{\ell+e_i} \prod_{j=1}^r I_{a_j}(\tau_j) 
	 = \Xi   \left( X_i  \, \widehat{\triangleright} X^{\ell} \right) \prod_{j=1}^r I_{a_j}(\tau_j).
\end{equs}
If  $ \sigma = I_a(\hat{\sigma}) $ then
\begin{equs}
	\Xi \sigma  X^{\ell} \prod_{j=1}^r I_{a_j}(\tau_j) &  = \Xi  I_a(\hat{\sigma}) X^{\ell} \prod_{j=1}^r I_{a_j}(\tau_j)
	\\ 
	& = \Xi   \sum_{m} {\ell \choose m}   X^{\ell - m}I_{a-m}(\hat{\sigma}) \prod_{j=1}^r I_{a_j}(\tau_j)
\end{equs}
where we have used \eqref{relation} repeatedly. In the end, one obtains:
\begin{equs} 
	\sigma \, \widehat{\triangleright}_l \, \tau = 
	I_b\left(	\Xi  \left(\sigma \, \widehat{\triangleright} \,  X^{\ell} \right) \prod_{j=1}^r I_{a_j}(\tau_j) + \Xi X \sum_{n=1}^r \left( \sigma \, \widehat{\triangleright} \, I_{a_n}(\tau_n) \right)  \prod_{j \neq n}  I_{a_j}(\tau_j) \right) 
\end{equs}
which is exactly the recursive definition of $  \widehat{\triangleright}  $.
	\end{proof}
Below, we illustrate Proposition~\ref{prop_left_grafting}. We first compute the left-most grafting and we use \eqref{relation} 
\begin{equs}
I_a(\Xi) \,	\widehat{\triangleright}_l \,
\begin{tikzpicture}[scale=0.2,baseline=0.1cm]
	\node at (0,0)  [dot,label= {[label distance=-0.2em]below: \scriptsize  $    $} ] (root) {};
	\node at (2,4)  [dot,label={[label distance=-0.2em]above: \scriptsize  $  $}] (right) {};
	\node at (-2,4)  [dot,label={[label distance=-0.2em]above: \scriptsize  $ $} ] (left) {};
	\draw[kernel1] (right) to
	node [sloped,below] {\small }     (root); \draw[kernel1] (left) to
	node [sloped,below] {\small }     (root);
	\node at (-1,2) [fill=white,label={[label distance=0em]center: \scriptsize  $ \Xi $} ] () {};
	\node at (1,2) [fill=white,label={[label distance=0em]center: \scriptsize  $ X_i $} ] () {};
\end{tikzpicture} = \begin{tikzpicture}[scale=0.2,baseline=0.1cm]
\node at (0,0)  [dot,label= {[label distance=-0.2em]below: \scriptsize  $     $} ] (root) {};
\node at (0,5)  [dot,label= {[label distance=-0.2em]above: \scriptsize  $     $} ] (center) {};
\node at (3,4)  [dot,label={[label distance=-0.2em]above: \scriptsize  $  $}] (right) {};
\node at (0,9)  [dot,label={[label distance=-0.2em]above: \scriptsize  $  $}] (rightc) {};
\node at (-3,4)  [dot,label={[label distance=-0.2em]above: \scriptsize  $ $} ] (left) {};
\draw[kernel1] (right) to
node [sloped,below] {\small }     (root); 
\draw[kernel1] (center) to
node [sloped,below] {\small }     (rightc); 
\draw[kernel1] (center) to
node [sloped,below] {\small }     (root); 
\draw[kernel1] (left) to
node [sloped,below] {\small }     (root);
\node at (-2,2) [fill=white,label={[label distance=0em]center: \scriptsize  $ \Xi $} ] () {};
\node at (2,2) [fill=white,label={[label distance=0em]center: \scriptsize  $ X_i $} ] () {};
\node at (0,7) [fill=white,label={[label distance=0em]center: \scriptsize  $ \Xi $} ] () {};
\node at (0,2.5) [fill=white,label={[label distance=0em]center: \scriptsize  $ a $} ] () {};
\end{tikzpicture} = \begin{tikzpicture}[scale=0.2,baseline=0.1cm]
\node at (0,0)  [dot,label= {[label distance=-0.2em]below: \scriptsize  $     $} ] (root) {};
\node at (0,5)  [dot,label= {[label distance=-0.2em]above: \scriptsize  $     $} ] (center) {};
\node at (3,4)  [dot,label={[label distance=-0.2em]above: \scriptsize  $  $}] (right) {};
\node at (3,8)  [dot,label={[label distance=-0.2em]above: \scriptsize  $  $}] (rightc) {};
\node at (-3,4)  [dot,label={[label distance=-0.2em]above: \scriptsize  $ $} ] (left) {};
\draw[kernel1] (rightc) to
node [sloped,below] {\small }     (right);
\draw[kernel1] (right) to
node [sloped,below] {\small }     (root); 
\draw[kernel1] (center) to
node [sloped,below] {\small }     (root); 
\draw[kernel1] (left) to
node [sloped,below] {\small }     (root);
\node at (-2,2) [fill=white,label={[label distance=0em]center: \scriptsize  $ \Xi $} ] () {};
\node at (3,6) [fill=white,label={[label distance=0em]center: \scriptsize  $ \Xi $} ] () {};
\node at (2,2) [fill=white,label={[label distance=0em]center: \scriptsize  $ a $} ] () {};
\node at (0,2.5) [fill=white,label={[label distance=0em]center: \scriptsize  $ X_i $} ] () {};
\end{tikzpicture} + \begin{tikzpicture}[scale=0.2,baseline=0.1cm]
\node at (0,0)  [dot,label= {[label distance=-0.2em]below: \scriptsize  $    $} ] (root) {};
\node at (2,4)  [dot,label={[label distance=-0.2em]above: \scriptsize  $  $}] (right) {};
\node at (2,8)  [dot,label={[label distance=-0.2em]above: \scriptsize  $  $}] (rightc) {};
\node at (-2,4)  [dot,label={[label distance=-0.2em]above: \scriptsize  $ $} ] (left) {};
\draw[kernel1] (right) to
node [sloped,below] {\small }     (root); 
\draw[kernel1] (right) to
node [sloped,below] {\small }     (rightc);
\draw[kernel1] (left) to
node [sloped,below] {\small }     (root);
\node at (-1,2) [fill=white,label={[label distance=0em]center: \scriptsize  $ \Xi $} ] () {};
\node at (2,6) [fill=white,label={[label distance=0em]center: \scriptsize  $ \Xi $} ] () {};
\node at (1,2) [fill=white,label={[label distance=0em]center: \scriptsize  $ \qquad a - e_i $} ] () {};
\end{tikzpicture}.
\end{equs}
On the other  hand, one has
\begin{equs}
	I_a(\Xi) \,	\widehat{\triangleright} \,
	\begin{tikzpicture}[scale=0.2,baseline=0.1cm]
		\node at (0,0)  [dot,label= {[label distance=-0.2em]below: \scriptsize  $  e_i  $} ] (root) {};
		\node at (0,4)  [dot,label={[label distance=-0.2em]above: \scriptsize  $  $}] (right) {};
		\draw[kernel1] (right) to
		node [sloped,below] {\small }     (root);
		\node at (0,2) [fill=white,label={[label distance=0em]center: \scriptsize  $ \Xi $} ] () {};
	\end{tikzpicture} = \begin{tikzpicture}[scale=0.2,baseline=0.1cm]
	\node at (0,0)  [dot,label= {[label distance=-0.2em]below: \scriptsize  $  e_i  $} ] (root) {};
	\node at (2,4)  [dot,label={[label distance=-0.2em]above: \scriptsize  $  $}] (right) {};
	\node at (2,8)  [dot,label={[label distance=-0.2em]above: \scriptsize  $  $}] (rightc) {};
	\node at (-2,4)  [dot,label={[label distance=-0.2em]above: \scriptsize  $ $} ] (left) {};
	\draw[kernel1] (right) to
	node [sloped,below] {\small }     (root); 
	\draw[kernel1] (right) to
	node [sloped,below] {\small }     (rightc);
	\draw[kernel1] (left) to
	node [sloped,below] {\small }     (root);
	\node at (-1,2) [fill=white,label={[label distance=0em]center: \scriptsize  $ \Xi $} ] () {};
	\node at (2,6) [fill=white,label={[label distance=0em]center: \scriptsize  $ \Xi $} ] () {};
	\node at (1,2) [fill=white,label={[label distance=0em]center: \scriptsize  $  a  $} ] () {};
\end{tikzpicture} +  \begin{tikzpicture}[scale=0.2,baseline=0.1cm]
	\node at (0,0)  [dot,label= {[label distance=-0.2em]below: \scriptsize  $    $} ] (root) {};
	\node at (2,4)  [dot,label={[label distance=-0.2em]above: \scriptsize  $  $}] (right) {};
	\node at (2,8)  [dot,label={[label distance=-0.2em]above: \scriptsize  $  $}] (rightc) {};
	\node at (-2,4)  [dot,label={[label distance=-0.2em]above: \scriptsize  $ $} ] (left) {};
	\draw[kernel1] (right) to
	node [sloped,below] {\small }     (root); 
	\draw[kernel1] (right) to
	node [sloped,below] {\small }     (rightc);
	\draw[kernel1] (left) to
	node [sloped,below] {\small }     (root);
	\node at (-1,2) [fill=white,label={[label distance=0em]center: \scriptsize  $ \Xi $} ] () {};
	\node at (2,6) [fill=white,label={[label distance=0em]center: \scriptsize  $ \Xi $} ] () {};
	\node at (1,2) [fill=white,label={[label distance=0em]center: \scriptsize  $ \qquad a - e_i $} ] () {};
\end{tikzpicture}.
\end{equs}
We conclude by using the identification of consecutive edges of the form $ X^k $ at the right location of the noise $ \Xi $ to node decoration $ k $ which gives:
\begin{equs}
\begin{tikzpicture}[scale=0.2,baseline=0.1cm]
	\node at (0,0)  [dot,label= {[label distance=-0.2em]below: \scriptsize  $     $} ] (root) {};
	\node at (0,5)  [dot,label= {[label distance=-0.2em]above: \scriptsize  $     $} ] (center) {};
	\node at (3,4)  [dot,label={[label distance=-0.2em]above: \scriptsize  $  $}] (right) {};
	\node at (3,8)  [dot,label={[label distance=-0.2em]above: \scriptsize  $  $}] (rightc) {};
	\node at (-3,4)  [dot,label={[label distance=-0.2em]above: \scriptsize  $ $} ] (left) {};
	\draw[kernel1] (rightc) to
	node [sloped,below] {\small }     (right);
	\draw[kernel1] (right) to
	node [sloped,below] {\small }     (root); 
	\draw[kernel1] (center) to
	node [sloped,below] {\small }     (root); 
	\draw[kernel1] (left) to
	node [sloped,below] {\small }     (root);
	\node at (-2,2) [fill=white,label={[label distance=0em]center: \scriptsize  $ \Xi $} ] () {};
	\node at (3,6) [fill=white,label={[label distance=0em]center: \scriptsize  $ \Xi $} ] () {};
	\node at (2,2) [fill=white,label={[label distance=0em]center: \scriptsize  $ a $} ] () {};
	\node at (0,2.5) [fill=white,label={[label distance=0em]center: \scriptsize  $ X_i $} ] () {};
\end{tikzpicture} =	\begin{tikzpicture}[scale=0.2,baseline=0.1cm]
		\node at (0,0)  [dot,label= {[label distance=-0.2em]below: \scriptsize  $  e_i  $} ] (root) {};
		\node at (2,4)  [dot,label={[label distance=-0.2em]above: \scriptsize  $  $}] (right) {};
		\node at (2,8)  [dot,label={[label distance=-0.2em]above: \scriptsize  $  $}] (rightc) {};
		\node at (-2,4)  [dot,label={[label distance=-0.2em]above: \scriptsize  $ $} ] (left) {};
		\draw[kernel1] (right) to
		node [sloped,below] {\small }     (root); 
		\draw[kernel1] (right) to
		node [sloped,below] {\small }     (rightc);
		\draw[kernel1] (left) to
		node [sloped,below] {\small }     (root);
		\node at (-1,2) [fill=white,label={[label distance=0em]center: \scriptsize  $ \Xi $} ] () {};
		\node at (2,6) [fill=white,label={[label distance=0em]center: \scriptsize  $ \Xi $} ] () {};
		\node at (1,2) [fill=white,label={[label distance=0em]center: \scriptsize  $  a  $} ] () {};
	\end{tikzpicture}. 
\end{equs}
\section{A post-Lie algebra for multi-indices}

Recall the Lie algebras of derivations on $\mathbb{R}[[z_{k}, 
z_{n}]]$ defined in \cite{LOT}, which are
\begin{equs}
\tilde{L} = \Big \langle \{ z^{\gamma}D^{(n)} \}_{[\gamma] \geq 0} \Big \rangle_{\mathbb{R}}
\end{equs}
which, equipped with the pre-Lie product $ \blacktriangleright $ defined in \cite{LOT} by
\begin{equs}
z^{\gamma}D^{(n)} \blacktriangleright z^{\gamma'}D^{(n')} = \sum_{\beta'}  
(z^{\gamma}D^{(n)})^{\gamma'}_{\beta'} z^{\beta'} D^{(n')}
\end{equs}
is also a pre-Lie algebra. We also consider
\begin{equs}
L = \Big \langle \{ z^{\gamma}D^{(n)} \}_{[\gamma] \geq 0} \cup \{ \partial_{i} \}_{i=0,...,d} \Big \rangle_{\mathbb{R}}.
\end{equs}
Here $\gamma$ is a multi-index and $n \in \mathbb{N}_{0}^{d+1}$. Note that we do not impose the condition $ |\gamma|_{\s} - |n|_{\s} > 0$. The Lie bracket $ [\cdot,\cdot] $ on $ L $ is defined by:
\begin{equation*} 
\begin{aligned}
& [ z^{\gamma}D^{(n)}, z^{\gamma'}D^{(n')} ]  = \sum_{\beta'}  
(z^{\gamma}D^{(n)})^{\gamma'}_{\beta'} z^{\beta'} D^{(n')} - \sum_{\beta}  
(z^{\gamma'}D^{(n')})^{\gamma}_{\beta} z^{\beta} D^{(n)} \\ & [\partial_i,\partial_j]  = 0, \quad [z^{\gamma}D^{(n)}, \partial_i] 
= n_i z^{\gamma}D^{(n-e_i)} - \sum_{\beta} (\partial_i)^{\gamma}_{\beta} z^{\beta} D^{(n)}.
\end{aligned}
\end{equation*}
where the $ \beta $ and  $ \beta' $ are such that $ [\beta] \geq  0, [\beta'] \geq 0 $.

\begin{definition}

The Lie algebra $L_{0}$ is the Lie algebra with underlying vector space $L$ and the Lie bracket $[ x, y]_{0}$ which is defined as $[ x, y]_{0} = 0$ for $x, y \in \tilde{L}$ and $x, y \in \langle \partial_{i} \rangle_{\mathbb{R}}$. We then define 
\begin{equation*}
[z^{\gamma}D^{(n)} , \partial_{i}]_{0} = n_{i}z^{\gamma}D^{(n-e_{i})}.
\end{equation*}
\end{definition}

\begin{remark}

This choice of bracket will directly translate to the fact that we can only partially symmetrize the basis of $U(L)$.

\end{remark}

\begin{remark}
	The space $\tilde{L}$ equipped with $ \blacktriangleright $ corresponds to an example of the natural pre-Lie algebra given in \cite[Sec. 2.1]{Burde}. But the adjunction of the operator $ \partial_i $ changes the structure and can be seen as a non-commutative generalisation. Indeed, the derivative operators $ \partial_i $ and $ D^{(n)} $ do not commute. 
	Such non-commutative relations appear in
	elementary quantum mechanics, see for example \cite[Eq. 10.7]{Wil}.
	\end{remark}

\begin{definition}

We define the product $x \ \widehat{\blacktriangleright} \ y = x \blacktriangleright y$ for all $x,y \in \tilde{L}$. We then set $\partial_{i} \ \widehat{\blacktriangleright} \ z^{\gamma}D^{(n)} = \ \partial_{i} z^{\gamma}D^{(n)}  $ and  $z^{\gamma}D^{(n)} \ \widehat{\blacktriangleright} \ \partial_{i} = 0$ otherwise.

\end{definition}

\begin{theorem}
The space $L$ equipped with the Lie bracket $[ x , y ]_{0}$ and the product $\widehat{\blacktriangleright}$ is a post-Lie algebra. Furthermore, the Lie bracket $[[x, y]] = [x, y]_0 + x  \widehat{\blacktriangleright} y - y \widehat{\blacktriangleright} x$ is equal to the original Lie bracket $[x, y]$ on $L$.
\end{theorem} 

\begin{proof}

One carefully checks that the axioms of a post-Lie algebra hold. The proof is similar to the one for decorated trees. More specifically the axioms for a post-Lie algebra are:
\begin{equs} 
x \, \widehat{\blacktriangleright} \, [y,z]_0 = [x \, \widehat{\blacktriangleright} \, y,z]_0 + [y, x \, \widehat{\blacktriangleright} \, z]_0
\end{equs}
and
\begin{equation*}
[x, y]_{0} \, \widehat{\blacktriangleright} \, z = a_{\widehat{\blacktriangleright}}(x,y,z) - a_{\widehat{\blacktriangleright}}(y,x,z).
\end{equation*}
As in the previous section, one easily checks that the first property is true. We prove the second property by distinguishing cases.

\vspace{5pt}

Case 1: If $x, y \in  \langle \partial_{i}  \rangle_{\mathbb{R}}$, then either $z \in \langle \partial_{i}  \rangle_{\mathbb{R}}$ and the property holds trivially, or $z \in \tilde{L}$ and we have $0$ on the left-hand side and only two non-zero terms on the right-hand side that cancel each other out. Hence, the property holds in this case.

Case 2: $x, y \in \tilde{L}$ and $z \in  \langle \partial_{i}  \rangle_{\mathbb{R}}$, then all terms are $0$ by definition of the product $\widehat{\blacktriangleright}$ and the property holds trivially.

Case 3: If $x, y, z \in \tilde{L}$, we notice that, in this case, the second property is equivalent to 
\begin{equation*}
[x, y]_{0} \, \widehat{\blacktriangleright} \, z = a_{{\blacktriangleright}}(x,y,z) - a_{\blacktriangleright}(y,x,z)
\end{equation*}
which is simply the pre-Lie property for the product $\blacktriangleright$.

Case 4: If $y, z \in \tilde{L}$ and $x \in  \langle \partial_{i}  \rangle_{\mathbb{R}}$, then from the left hand side, we have:
\begin{equation*}
[x,y]_0 \ \widehat{\blacktriangleright} \ z = n_{i}z^{\gamma}D^{(n-e_i)} \blacktriangleright z^{\gamma '}D^{(n')}.
\end{equation*}
From the right hand side, the following terms remain:
\begin{equation*}
\begin{aligned}
a_{\widehat{\blacktriangleright}}(x,y,z)  & = z^{\gamma}D^{(n)} \, \widehat{\blacktriangleright} \, ( X_i \, \widehat{\blacktriangleright} \, z^{\gamma '}D^{(n')}) - (z^{\gamma}D^{(n)} \, \widehat{\blacktriangleright} \, X_i) \, \widehat{\blacktriangleright} \, z^{\gamma '}D^{(n')}
\\ & =  z^{\gamma}D^{(n)} \, \widehat{\blacktriangleright} \, (\partial_{i} z^{\gamma '}D^{n'})
\end{aligned}
\end{equation*}
because $z^{\gamma}D^{(n)} \, \widehat{\blacktriangleright} \, \partial_{i} = 0  $. Then
\begin{equation*}
\begin{aligned}
 a_{\widehat{\blacktriangleright}}(y,x,z) & = \partial_{i} \, \widehat{\blacktriangleright} \, (z^{\gamma}D^{(n)})  \, \widehat{\blacktriangleright} \, z^{\gamma '}D^{(n')}) - (\partial_{i} \, \widehat{\blacktriangleright} \, z^{\gamma}D^{(n)}) \, \widehat{\blacktriangleright} \, z^{\gamma '}D^{(n')} \\
 & = \ \partial_{i}(z^{\gamma}D^{(n)} \blacktriangleright z^{\gamma '}D^{(n')} ) - (\partial_{i}z^{\gamma}D^{(n)})\blacktriangleright z^{\gamma '}D^{(n')}.
 \end{aligned}
\end{equation*}
Finally, one obtains
\begin{equs}
a_{\widehat{\blacktriangleright}}(x,y,z)  - a_{\widehat{\blacktriangleright}}(y,x,z) &  =  z^{\gamma}D^{(n)} \, \blacktriangleright \,  (\partial_{i} z^{\gamma '}D^{(n'}) - \partial_{i}(z^{\gamma}D^{(n)} \blacktriangleright z^{\gamma '}D^{(n')} )  
  \\ & + (\partial_{i}z^{\gamma}D^{(n)})\blacktriangleright z^{\gamma '}D^{(n')}
\end{equs}
which is equal to $n_{i}z^{\gamma}D^{(n-e_i)} \blacktriangleright z^{\gamma '}D^{(n')}$ after unwinding definitions, as desired.

\vspace{5pt}

 Since any remaining cases are covered by symmetry we have a post-Lie algebra. To finish the proof, we notice that the second claim of the lemma is immediate from the definitions of the brackets and the product $\widehat{\blacktriangleright}$.

\end{proof}

\begin{remark} \label{post_Lie_morphism} The map $ \Psi $ defined \eqref{pre_Lie_isomorphism} can be extended to a pre-Lie morphism between the vector space $ 	\tilde{\mathcal{V}}_0 $ given by
	\begin{equs}
	\tilde{\mathcal{V}}_0 & = \Big  \langle \{ I_a(\tau), \, a \in \mathbb{N}^{d+1}, \, \tau \in \mathfrak{T}_0 \} \Big \rangle_{\mathbb{R}}
	\end{equs}
	 and $ \tilde{L} $. We denote this extension by $ \hat{\Psi} $ and it is given for every $ I_a(\sigma) \in \tilde{\mathcal{V}}_0 $ by:
	 \begin{equs}
	 	\hat{\Psi}(I_a(\sigma)) =\Psi(\sigma) \frac{D^{(a)}}{a!}.
	 \end{equs}  
 One also has
 \begin{equs}
 	\hat{\Psi}(  	I_a(\sigma) \, \widehat{\curvearrowright} \, I_b(\tau)   )  = 	\hat{\Psi}(  	I_a(\sigma)  )\blacktriangleright \hat{\Psi}(I_b(\tau))
 \end{equs}
where the pre-Lie product on the decorated trees is the one given by \eqref{pre_lie_bis}. To obtain a post-Lie morphism, one has to extend the map $ \hat{\Psi} $ on the $ X_i $ by setting:
\begin{equs}
	\hat{\Psi}(X_i) = \partial_i. 
\end{equs}
For example, one easily checks that 
\begin{equs}
	\hat{\Psi}(X_i \, \widehat{\triangleright} \, I_{a}(\sigma)) = \hat{\Psi}(\uparrow^{i} I_{a}(\sigma)) = \partial_{i} \,  \widehat{\blacktriangleright} \, \Psi(\sigma)\frac{D^{(a)}}{a!} = \hat{\Psi}(X_i) \, \widehat{\triangleright} \, \hat{\Psi}(I_{a}(\sigma))
\end{equs}
where here $ \uparrow^{i} $ applied to a decoration $ \prod_{j=1}^{k} X^{\ell_j} $ gives:
\begin{equs}
	\uparrow^{i}  \prod_{j=1}^{k} X^{\ell_j} =  X_i \prod_{j=1}^{k} X^{\ell_j} + \sum_{j=1}^k X^{\ell_{j} + e_i} \prod_{j' \neq j} X^{\ell_{j'}}
\end{equs}
Similarly for the bracket, one gets:
\begin{equs}
	\hat{\Psi}([X_i, I_{a}(\sigma)]) = \hat{\Psi}(I_{a-e_{i}}(\sigma)) =  \Psi(\sigma)\frac{D^{(a-e_{i})}}{(a-e_i)!} = [\hat{\Psi}(X_i), \hat{\Psi}(I_{a}(\sigma))].
\end{equs}
Due to the condition $ [\gamma] \geq 0 $, for any element of $ \tilde{L} $, there exists a pre-image in $ \mathcal{V}_0 $ via
$ \hat{\Psi} $.
	\end{remark}

We are now able to use the main theorem in \cite{ELM} which generalizes the Guin-Oudom theorem on the Lie enveloping algebra of a pre-Lie algebra to the case of a post-Lie algebra in order to obtain the following result:

\begin{theorem}
\label{main_result_multi_indices}
One has a Hopf algebra isomorphism of $U(L)$ with $(U(L_0), \star, \Delta)$. Furthermore, this isomorphism immediately gives us a partially symmetric basis for $U(L)$ as the one obtained in \cite{LOT}.

\end{theorem}

\begin{proof}
The first part is obtained directly by using the result in \cite{ELM}. Then, for the second part, we simply notice that $[x, y]_0$ vanishes on $\tilde{L}$ which results in the relation $xy = yx$ holding for $x, y \in \tilde{L}$ inside $U(L)$. This, together with the Poincare-Birkhoff-Witt theorem, gives us a basis for $U(L)$ that is symmetric with respect to the elements of the basis of $\tilde{L}$.
\end{proof}

\begin{remark}
	Using the post-Lie morphism $ \hat{\Psi} $, one has
	a direct connection between the product on multi-indices and the one on decorated trees. Indeed, the map $ \hat{\Psi} $ is in particular a Lie morphism for $[.,.]$ and so has a natural extension to $U(\mathfrak{g})$. Hence, it commutes with the construction in Section~\ref{section::Post_Lie}. By virtue of this, it extends to a morphism of Hopf algebras and so one has for  $ A, B \in U(L_0)$ 
	\begin{equs}
		A  \star B = 	\hat{\Psi}(x)   \star
		 \hat{\Psi}(y) = \hat{\Psi} (  x   *
		 y )
		\end{equs}
	where $ x, y \in U(\mathcal{V}_0)  $ are such that $ A = \hat{\Psi}(x) $ and $ B = \hat{\Psi}(y) $.
	\end{remark}

\end{document}